\documentclass[12pt,reqno] {amsart}
\usepackage{amsfonts,amssymb,amsmath,amsthm}
\usepackage{mathrsfs}
\usepackage{url}
\usepackage{enumerate}
\usepackage{xcolor}
\usepackage{todonotes}
\usepackage{ulem}

\newtheorem{theorem}{Theorem}[section]
\newtheorem*{theorem*}{Theorem}
\newtheorem{lemma}[theorem]{Lemma}

\newtheorem{prop}[theorem]{Proposition}
\newtheorem{cor}[theorem]{Corollary}
\theoremstyle{definition}
\newtheorem{defn}[theorem]{Definition}
\newtheorem{rem}[theorem]{Remark}

\newtheorem{example}[theorem]{Example}

\author[M. Bhattacharjee]{Monojit Bhattacharjee}
\author[R. Gupta]{Rajeev Gupta}
\author[V. Venugopal]{Vidhya Venugopal}

\address[M. Bhattacharjee]{Department of Mathematics \\
Birla Institute of Technology and Science K K Birla Goa Campus, India
}

\address[R. Gupta]{School of Mathematics and Computer Science\\
Indian Institute of Technology Goa, India}

\address[V. Venugopal]{Department of Mathematics \\
Birla Institute of Technology and Science K K Birla Goa Campus, India
 }

\begin{document}
\title[Dirichlet type spaces and Wandering Subspace Property]{Dirichlet type spaces in the unit bidisc and Wandering Subspace Property for operator tuples}
\maketitle
\begin{abstract}
  In this article, we define Dirichlet-type space $\mathcal{D}^{2}(\boldsymbol{\mu})$ over the bidisc $\mathbb D^2$ for any measure $\boldsymbol{\mu}\in\mathcal{P}\mathcal{M}_{+}(\mathbb T^2).$ We show that the set of polynomials is dense in $\mathcal{D}^{2}(\boldsymbol{\mu})$ and the pair $(M_{z_1}, M_{z_2})$ of multiplication operator by co-ordinate functions on $\mathcal{D}^{2}(\boldsymbol{\mu})$ is a pair of commuting $2$-isometries. Moreover, the pair $(M_{z_1}, M_{z_2})$ is a left-inverse commuting pair in the following sense: $L_{M_{z_i}} M_{z_j}=M_{z_j}L_{M_{z_i}}$ for $1\leqslant i\neq j\leqslant n,$ where $L_{M_{z_i}}$ is the left inverse of $M_{z_i}$ with $\ker L_{M_{z_i}} =\ker M_{z_i}^*$, $1\leqslant i \leqslant n$. Furthermore, it turns out that, for the class of left-inverse commuting tuple $\boldsymbol T=(T_1, \ldots, T_n)$ acting on a Hilbert space $\mathcal{H}$, the joint wandering subspace property is equivalent to the individual wandering subspace property. As an application of this, the article shows that the class of left-inverse commuting pair with certain splitting property is modelled by the pair of multiplication by co-ordinate functions $(M_{z_1}, M_{z_2})$ on  $\mathcal{D}^{2}(\boldsymbol{\mu})$ for some $\boldsymbol{\mu}\in\mathcal{P}\mathcal{M}_{+}(\mathbb T^2).$

\end{abstract}

\section{Introduction}

{ An important aspect of the representation theorem of Richter \cite[Theorem 5.1]{Richter1} is its implication that understanding a class of non-self adjoint bounded linear operators on a Hilbert space requires the study of (analytic) function theory. For instance, if $T$ is a cyclic analytic $2$-isometry on a Hilbert space $\mathcal{H}$ (that is, $\|T^2h\|^2 - 2 \|Th\|^2 + \|h\|^2 = 0$ for all $h \in \mathcal{H}$ and $\cap_{n \geqslant 0} T^n \mathcal{H} = \{0\} $) then there exists a $\mathcal{H} \ominus T \mathcal{H}$ 
($=: \mathcal{E}$)-valued positive Borel measure $\mu$ on the unit circle $\mathbb T$ such that $T$ is unitarily equivalent to $M_z$ on $\mathcal{D}_{\mu}(\mathbb D),$ where $\mathcal{D}_{\mu}(\mathbb D)$ denotes the corresponding Dirichlet-type space on the open unit disc $\mathbb D$, as introduced in \cite{Richter1}. As $T$ is cyclic, here the dimension of $\mathcal{E}$ is $1.$
By dropping the cyclicity condition on $T,$
in \cite{Olo}, Olofsson identifies the unitary equivalent class of analytic $2$-isometries by the shifts on an analytic function space over $\mathbb D$ which intrinsically depends on the $\mathcal{E}$-valued spectral measure. Moreover, it is observed in \cite[Theorem 1]{Richter2} that if $T$ is analytic $2$-isometry on a Hilbert space $\mathcal H$ then $\mathcal H=\overline{\mathrm{span}}\{T^n\mathcal E: n\geqslant 0\}.$


This indicates that, on one hand, any analytic $2$-isometry is unitarily equivalent to a shift on the Dirichlet-type space, and on the other hand, it possess the wandering subspace property (defined below).


In this context, it is noteworthy that the representation theorems for the class of analytic 2-isometries available in \cite{Richter1, Olo} have become classical results. The wandering subspace property for a class of bounded operators and corresponding model spaces have had a profound influence on the development of operator theory and function theory in one variable. A key goal of multivariable operator theory is to identify classes of commuting tuples of bounded linear operators on a Hilbert space that possess the wandering subspace property and, consequently, determine their unitary equivalence class in terms of co-ordinatewise multiplication operators on an analytic function space over the polydisc $\mathbb D^n.$
}

We pause for a moment to introduce some notations and recall some useful definitions for what follows in this article. Let $\mathbb N$ and $\mathbb Z$ denote the set of all positive integers and the set of all integers respectively. For each $n\in\mathbb N,$ let $\mathbb Z^n_{\geqslant 0}$ be the set of all $n$-tuples of non-negative integers.

Let $\boldsymbol{T}=(T_1,\ldots,T_n)$ be an $n$-tuple of commuting bounded linear operators on a Hilbert space $\mathcal{H}$. A closed subspace $\mathcal{W} \subseteq \mathcal{H}$ is said to be a \textit{wandering subspace} of $\boldsymbol{T}$ if 
\[
\mathcal{W}  \perp T_1^{k_1}\cdots T_n^{k_n} \mathcal{W} \quad \quad \text{for all non-zero}~ (k_1,\ldots, k_n) \in \mathbb Z_{\geqslant 0}^n. 
\] 
Following Halmos \cite{Hal}, a wandering subspace $\mathcal{W}$ of $\boldsymbol{T}$ in $\mathcal{H}$ is said to be a \textit{generating wandering subspace} for $\boldsymbol{T}$ if 
\[ 
\mathcal{H} = \overline{\mathrm{span}} \big\{T_1^{k_1}\cdots T_n^{k_n} \mathcal{W}~:~ (k_1,\ldots, k_n) \in \mathbb Z_{\geqslant 0}^n \big\}.
\]
We say an $n$-tuple of operators $\boldsymbol{T}=(T_1,\ldots,T_n)$ acting on a Hilbert space $\mathcal{H}$ has \textit{joint wandering subspace property} and \textit{individual wandering subspace property} provided $\boldsymbol{T}$ has a generating wandering subspace in $\mathcal{H}$ and each $T_i$ for $i=1,\ldots,n,$ has generating wandering subspace, respectively. We say that an $n$-tuple of operators $\boldsymbol{T}=(T_1,\ldots,T_n)$ acting on a Hilbert space $\mathcal{H}$ is cyclic if there exists a vector $x_0\in\mathcal H$ such that $$\mathcal{H} = \overline{\mathrm{span}} \big\{T_1^{k_1}\cdots T_n^{k_n} x_0~:~ (k_1,\ldots, k_n) \in \mathbb Z_{\geqslant 0}^n \big\}.$$

One of the major motivation behind studying wandering subspaces is their natural connection with invariant subspaces. The existence of {wandering subspace property} for the restriction of multiplication by the co-ordinate function $M_z$ to any invariant subspace of the Hardy space $H^2(\mathbb D)$,  
the Bergman space $A^2(\mathbb D)$, and the Dirichlet-type space 
$\mathcal{D}_{\mu}(\mathbb{D})$ for any finite positive Borel measure $\mu$ on the unit circle $\mathbb T$ are shown by Beurling  in \cite{B}, by Aleman, Richter and Sundberg in \cite{ARS}, and by Richter in \cite{Richter2}, respectively. Later on, Shimorin proved the existence of this property for an abstract class of left-invertible operators (see \cite{Shimorin}).  

In multi-variable operator theory, characterization of invariant subspaces for $n$-tuples of operators is one of the fascinating and challenging topic from the last few decades. Among these, characterizing $n$-tuples of commuting bounded linear operators on a Hilbert space which possess the {joint wandering subspace property} has emerged as one of the most significant and intriguing questions. 
In \cite{Rudin-poly}, Rudin demonstrates that the existence of the {wandering subspace property} of general invariant subspace fails even for a pair of commuting isometries. 
In particular, he constructs an invariant subspace of Hardy space over bidisc not having {joint wandering subspace property}. 
For similar kind of examples on analytic functions over Euclidean unit ball, we refer the reader to \cite{BEKS}. Consequently, a straightforward generalization of the single variable result is not possible.  
Meanwhile, the existence of {joint wandering subspace property} for $n$-tuple of bounded linear operators $(T_1,\ldots,T_n)$ acting on a Hilbert space $\mathcal{H}$ with individual wandering subspace property has been established under an additional assumption of doubly commutativity, that is, $T_iT_j = T_jT_i$ and $T_i^*T_j = T_jT_i^*$ for all $1 \leqslant i < j \leqslant n$.
For more details, we refer the reader to \cite{Ma, RT, SSW, ABJS} and references therein. 

{In this article, we establish the analogous result for a much larger class of operator tuples namely left-inverse commuting tuple defined as below.}
\begin{defn}[Left-inverse commuting tuple]
        A tuple $(T_1,\ldots, T_n)$ of left invertible commuting operators is said to be a \textit{left-inverse commuting tuple} if     \begin{eqnarray}\label{Cauchy-dual commuting pairs}
      L_iT_j=T_jL_i \quad \mbox{for } 1\leqslant i\neq j\leqslant n,
  \end{eqnarray}
  where $L_j$ denotes the adjoint of the Cauchy-dual of $T_j$, that is, $L_j = (T_j^*T_j)^{-1}T_j^*.$ 
    \end{defn}
Note that, for each $1\leqslant j \leqslant n,$ $L_j$ is a left inverse of $T_j$ in the above definition. 
{The inclusion of the class of doubly commuting tuples inside the class of left-inverse commuting tuples is automatic. Furthermore, in Example \ref{counter example - weighted shift} and Example \ref{Example-not doubly commuting}, it is demonstrated that any operator tuple in the class of operator-valued weighted shifts with invertible weights, as well as any operator tuple in the class of tuple of multiplication by co-ordinate functions on Dirichlet-type spaces over the bidisc introduced in \cite{Sameer et. al.}, belong to that of left-inverse commuting tuples but is not necessarily doubly commuting.

} One of the main results of this article is the following theorem, which is also the content of Theorem \ref{wandering subspace property-abstract pair} in what follows:
\begin{theorem}
    Let $\boldsymbol T=(T_1, \ldots, T_n)$ be a left-inverse commuting tuple on a Hilbert space $\mathcal H$. Then the following are equivalent: \begin{itemize}
        \item[(a)] the tuple $\boldsymbol T$ has joint wandering subspace property.
        \item[(b)] the tuple $\boldsymbol T$ has individual wandering subspace property. 
    \end{itemize}
\end{theorem}
This significantly extends our understanding of the class of commuting tuple of bounded linear operators on a Hilbert space that possess the wandering subspace property. In particular, we prove a multi-variable generalization of Shimorin's result \cite{Shimorin}, which turns out to be a generalization of all aforesaid results in this regard. 
Theorem \ref{wandering subspace property-abstract pair} in this article, in particular, proves that if $(T_1,\ldots, T_n)$ is a left-inverse commuting tuple of analytic $2$-isometries on a Hilbert space $\mathcal{H}$ then $\cap_{i=1}^{n}\ker T_i^*$ serves as a wandering subspace for the the tuple $(T_1,\ldots, T_n).$ 
  In \cite[Corollary 2.4]{ABJS}, authors have proven that if $(T_1,\ldots, T_n)$ is a doubly commuting tuple of bounded analytic $2$-isometries then $\cap_{i=1}^{n}\ker T_i^*$ serves as a wandering subspace for $(T_1,\ldots, T_n).$ 
  Note that if a tuple $(T_1,\ldots, T_n)$ is doubly commuting then it necessarily satisfies \eqref{Cauchy-dual commuting pairs}. In this sense, Theorem \ref{wandering subspace property-abstract pair} generalizes \cite[Corollary 2.4]{ABJS}. 
  In Example \ref{counter example - weighted shift} and Example \ref{Example-not doubly commuting}, we present classes of examples of analytic $2$-isometric left-inverse commuting tuple $(T_1,\ldots, T_n)$ which are not doubly commuting. 

Along with this, we develop a natural bidisc counter part of the existing theory of the Dirichlet-type spaces over unit disc introduced by Richter \cite{Richter2, Richter1} with an explicit description of the norm. 
We also show that it becomes a model space for a class of pair of cyclic analytic $2$-isometries, see Theorem \ref{Model theorem}. It should be noted here that the model theorem Theorem \ref{Model theorem} in the picture can not be obtained from the model theorem of Richter for cyclic analytic $2$-isometries \cite[Theorem 5.1]{Richter1}, as $T_1$ and $T_2$ individually are not cyclic, {and splitting property in Theorem \ref{Model theorem} is very crucial for it to work}. With the help of the model theorem, we establish a connection between the doubly commuting pairs of analytic $2$-isometries and the class of analytic $2$-isometries introduced in this article.


\section{Dirichlet-type space over bidisc}\label{Dirichlet-type space over bidisc}
Let $\mathcal M_+(\mathbb T)$ denote the set of all positive finite Borel measures on the unit circle $\mathbb T.$ Let $\mathcal{PM}_{+}(\mathbb T^2)$ denote the set of all measures $\boldsymbol \mu$ on $\mathbb T^2$ such that $\boldsymbol \mu=\mu_1\times \mu_2,$ where $\mu_1$ and $\mu_2$ are positive finite Borel measures on $\mathbb T,$ i.e.
    \[
\mathcal{PM}_{+}(\mathbb T^2): = \left\{\boldsymbol{\mu} = \mu_1 \times \mu_2: \mu_1, \mu_2\in \mathcal M_+(\mathbb T)\right\}. \] 
    We define the Poisson kernel over $\mathbb D^2$ as a product of Poisson kernels over $\mathbb D,$ see also \cite[Chapter 2]{Rudin-poly}. More precisely, for any $z=(z_1,z_2) \in \mathbb D^2$ and  $\xi = (\xi_1,\xi_2) \in \mathbb T^2$, $\boldsymbol P(z, \xi) := P(z_1,\xi_1) ~ P(z_2,\xi_2),$
       where 
       $$P(z_i,\xi_i) = \frac{(1-|z_i|^{2})}{|\xi_i - z_i|^{2}}, \quad i=1,2.$$ 
       For any $\boldsymbol{\mu} \in \mathcal{PM}_{+}(\mathbb T^2)$, we define the \textit{Poisson integral} $\boldsymbol P_{\boldsymbol{\mu}}(z_1,z_2) := P_{\mu_1}(z_1)$ $P_{\mu_2}(z_2),$ for $(z_1,z_2)\in\mathbb D^2$, where 
       $$P_{\mu_i}(z_i) = \frac{1}{2\pi} \int_0^{2\pi} P(z_i,e^{it}) d\mu_i(t), \quad i=1,2.$$
Let $\mathcal{O}(\mathbb{D}^{2})$ denote the set of all complex valued holomorphic functions on the unit bidisc $\mathbb{D}^2.$
    For any $f \in \mathcal{O}(\mathbb{D}^{2})$ and for $\boldsymbol{\mu} \in \mathcal{PM}_{+}(\mathbb T^2),$ define
    \begin{eqnarray}\label{seminorm-first integral}
        D_{\boldsymbol{\mu}, 1}^{(2)}(f) &:= \lim_{r \rightarrow{1}^-} \frac{1}{2\pi}\iint_{\mathbb{D} \times \mathbb{T}} \left|\frac{\partial{f}}{\partial{z_1}} (z_1, re^{it})\right|^{2} P_{\mu_1}(z_1)~ dA(z_1)~ dt,    
        \end{eqnarray}
    \begin{eqnarray}\label{seminorm-second integral}  D_{\boldsymbol{\mu}, 2}^{(2)}(f) &:= \lim_{r \rightarrow{1}^-} \frac{1}{2\pi}\iint_{\mathbb{T} \times \mathbb{D}} \left|\frac{\partial{f}}{\partial{z_2}} (re^{it},z_2)\right|^{2} P_{\mu_2}(z_2)~ dA(z_2)  ~dt,    
    \end{eqnarray}
        and 
    \begin{eqnarray}\label{seminorm-third integral}
        D_{\boldsymbol{\mu}, 3}^{(2)}(f) &:= \iint_{\mathbb{D}^2} \left| \frac{\partial^2{f}}{\partial{z_1}\partial{z_2}} (z_1,z_2)\right|^{2} \boldsymbol P_{\boldsymbol{\mu}}(z_1,z_2) ~ dA(z_1) dA(z_2),    
        \end{eqnarray}    
     where $dA$ denotes the normalized Lebesgue area measure on the unit disc $\mathbb D$. 
    Wherever needed, we shall use the short hand notations $\partial_1 f, \partial_2 f$ and $\partial_1\partial_2 f$ in place of $\frac{\partial{f}}{\partial{z_1}}, \frac{\partial{f}}{\partial{z_2}}$ and $\frac{\partial^2{f}}{\partial{z_1}\partial{z_2}}$ respectively. Whenever it is convenient, we shall denote $\boldsymbol{z}=(z_1,z_2)$ and $d\boldsymbol{A}(\boldsymbol z)=dA(z_1)dA(z_2).$
    We define the Dirichlet-type integral over $\mathbb D^2$ as 
    \begin{align*}
     D^{(2)}_{\boldsymbol{\mu}}(f):=  D_{\boldsymbol{\mu}, 1}^{(2)}(f)+ D_{\boldsymbol{\mu}, 2}^{(2)}(f)+ D_{\boldsymbol{\mu}, 3}^{(2)}(f).
    \end{align*}
    We define the Dirichlet-type space on the unit bidisc with respect to the non-zero measure $\boldsymbol{\mu}$ to be the following space of analytic functions 
\begin{align*}
     \mathcal{D}^{2}(\boldsymbol{\mu}): = \{f\in \mathcal{O}(\mathbb{D}^{2}):   D^{(2)}_{\boldsymbol{\mu}}(f) < \infty \}.
\end{align*}
   For $\boldsymbol{\mu} =0,$ we set $\mathcal{D}^{2}(\boldsymbol{\mu)} = H^2(\mathbb D^2),$ the Hardy space over the unit bidisc. The function  $D^{(2)}_{\boldsymbol{\mu}}(\cdot)$ defines a semi-norm on $\mathcal{D}^{2}(\boldsymbol{\mu}).$
    In Lemma \ref{Dirichlet contained in Hardy}, we prove that Dirichlet-type space over unit bidisc with respect to any product positive measure $\boldsymbol{\mu}$ on $\mathbb T^2$ is contained in the Hardy space over unit bidisc as a subset. With this information, we can convert the semi-norm $D^{(2)}_{\boldsymbol{\mu}}(\cdot)$ into a norm in a traditional way by adding the norm borrowed from $H^2(\mathbb D^2)$ to it. More precisely, 
    for any $f \in \mathcal{D}^{2}(\boldsymbol{\mu}),$ we define the norm  $\|f\|_{\boldsymbol{\mu}}$ by setting
\begin{align*}
     \|f\|_{\boldsymbol{\mu}}^{2}= \|f\|_{{H}^{2}(\mathbb{D}^{2})}^{2} + D^{(2)}_{\boldsymbol{\mu}}(f). 
\end{align*}
It is observed in Theorem \ref{Dirichlet spaces are RKHS} that for any $\boldsymbol{\mu}\in\mathcal{PM}_{+}(\mathbb T^2)$, the space $(\mathcal{D}^{2}(\boldsymbol{\mu}), \|\cdot\|_{\boldsymbol{\mu}})$ is a reproducing kernel Hilbert space. 
When $\boldsymbol{\mu}$ is the Lebesgue measure on $\mathbb T^2$, then $\mathcal{D}^{2}(\boldsymbol{\mu})$ turns out to be the classical Dirichlet space over bidisc and is denoted by $\mathcal{D}^{2}$. The associated  kernel function $K_{\mathcal{D}^{2}}$ is given by
\begin{equation}\label{kernel-classical Dirichlet over bidisc}    K_{\mathcal{D}^{2}}\big((z_1,z_2),(w_1,w_2)\big)= 
    K_{\mathcal D}(z_1,w_1) K_{\mathcal D}(z_2,w_2), \quad (z_1,z_2),(w_{1},w_{2}) \in \mathbb{D}^{2},
\end{equation} 
    where $K_{\mathcal D}$ is the reproducing kernel for the classical Dirichlet space $\mathcal D$ on the unit disc $\mathbb D$. 

    In a recent article \cite{Sameer et. al.}, for any $\mu_1,\mu_2\in\mathcal M_+(\mathbb T)$, the Dirichlet-type space $\mathcal{D}(\mu_1,\mu_2)$ over the bidisc is defined. In their formulation of $\mathcal{D}(\mu_1, \mu_2)$, they considered the semi-norm $\mathcal{D}_{\mu_1,\mu_2}(\cdot)$, which is determined solely by the terms specified in equations \eqref{seminorm-first integral} and \eqref{seminorm-second integral}. 
    The primary objective of their article was to analyze and propose a suitable model for a class of toral $2$-isometries. By adopting their chosen norm, the authors in \cite{Sameer et. al.} achieve a squared norm for a function $f$ in their version of the classical Dirichlet space over the bidisc given by $\sum_{m,n}|\hat{f}(m,n)|^2(n+m+1)$. 
    This differs from that of  $\mathcal{D} \otimes \mathcal{D}$, which might be a very natural choice for the classical Dirichlet space over the bidisc.
Our motivation to introduce the quantity in equation \eqref{seminorm-third integral} is to restore this property, as indicated in equation \eqref{kernel-classical Dirichlet over bidisc}. Consequently, the class of commuting tuples $(T_1,T_2)$ that we investigate in this article differs significantly from those studied in \cite{Sameer et. al.}. 
Throughout this article, we occasionally refer to results obtained in \cite{Sameer et. al.} or incorporate some of their arguments. In such instances, we explicitly acknowledge the source.

  The multiplication by co-ordinate functions $z_1$ and $z_2$ are bounded analytic $2$-isometries on $\mathcal{D}^{2}(\boldsymbol{\mu}).$ These operators further satisfy \eqref{Cauchy-dual commuting pairs}. It turns out that the set of polynomials densely sits in the Hilbert space $(\mathcal{D}^{2}(\boldsymbol{\mu}), \|\cdot\|_{\boldsymbol{\mu}})$. Moreover, we observe in Lemma \ref{inner product- monomials - in term of single variable} that 
  \begin{eqnarray*}
      \langle z_1^m z_2^n, z_1^p z_2^q \rangle_{\boldsymbol{\mu}} = \langle z_1^m, z_1^p \rangle_{\mathcal{D}(\mu_1)} \langle z_2^n, z_2^q \rangle_{\mathcal{D}(\mu_2)},
  \end{eqnarray*}
  for any $m,n,p,q\in\mathbb Z_{\geqslant 0}.$ 
  Here $\mathcal D(\mu_j)$ denotes the Dirichlet-type space over the unit disc as defined by Richter in \cite{Richter1}.
  
  Furthermore, in Theorem \ref{Model theorem}, we prove that if $(T_1,T_2)$ is a left-inverse commuting pair of analytic $2$-isometries with 
  \begin{eqnarray*}
      \langle T_1^m T_2^n x_0, T_1^p T_2^q x_0 \rangle = \langle T_1^m x_0, T_1^p x_0 \rangle \langle T_2^n x_0, T_2^q x_0\rangle, \quad m,n,p,q\in\mathbb Z_{\geqslant 0},
  \end{eqnarray*}
  for some unit vector $x_0\in\ker T_1^*\cap \ker T_2^*,$   then the pair $(T_1,T_2)$ is unitarily equivalent to $(M_{z_1},M_{z_2})$ on $\mathcal{D}^{2}(\boldsymbol{\mu)}$ for some $\boldsymbol{\mu}\in \mathcal P\mathcal M_+(\mathbb T^2).$

\section{A few properties of $\mathcal{D}^{2}(\boldsymbol{\mu})$} 
In what follows, it will be useful to set the following notations: for any $R, R_1, R_2\in (0,1),$ set
    \begin{align*}
       D_{\boldsymbol{\mu}, 1}^{(2)}(f,R) &:= \lim_{r \rightarrow{1}^-} \frac{1}{2\pi}\iint_{R\mathbb{D} \times \mathbb{T}} \left|\partial_1{f}(z_1, re^{it})\right|^{2} P_{\mu_1}(z_1)~ dA(z_1) ~dt\\
        D_{\boldsymbol{\mu}, 2}^{(2)}(f, R) &:= \lim_{r \rightarrow{1}^-} \frac{1}{2\pi}\iint_{\mathbb{T} \times R\mathbb{D}} \left|\partial_2{f}(re^{it},z_2)\right|^{2} P_{\mu_2}(z_2)~ dA(z_2)  ~dt\\
         D_{\boldsymbol{\mu}, 3}^{(2)}(f, R_1, R_2) &:= \iint_{R_1\mathbb{D}\times R_2\mathbb{D}} \left| \partial_1\partial_2{f}(z_1,z_2)\right|^{2} \boldsymbol{P}_{\boldsymbol{\mu}}(z_1,z_2) ~ dA(z_1) dA(z_2).
    \end{align*}
    If $R_1=R_2=R,$ then $D_{\boldsymbol{\mu}, 3}^{(2)}(f, R_1, R_2)$ will simply be denoted by $D_{\boldsymbol{\mu}, 3}^{(2)}(f, R).$ 

    The norm on $\mathcal{D}^{2}(\boldsymbol{\mu})$ satisfies the parallelogram identity, i.e.,   $\|f+g\|^2_{\boldsymbol{\mu}}+\|f-g\|^2_{\boldsymbol{\mu}} = 2(\|f\|^2_{\boldsymbol{\mu}}+\|g\|^2_{\boldsymbol{\mu}}),$
    and therefore $(\mathcal{D}^{2}(\boldsymbol{\mu}), \|\cdot\|_{\boldsymbol{\mu}})$ is an inner product space. We shall denote the inner product of $\mathcal{D}^{2}(\boldsymbol{\mu})$ by $\langle \cdot , \cdot \rangle_{\boldsymbol \mu}$. Thus for $f,g \in \mathcal{D}^{2}(\boldsymbol{\mu})$, we set 
\begin{equation*}
    \langle f, g \rangle_{\boldsymbol{\mu}}= \langle f,g \rangle_{H^2(\mathbb D^2)}+\langle f,g \rangle_{\boldsymbol{\mu},1}+ \langle f,g \rangle_{\boldsymbol{\mu},2}+\langle f,g \rangle_{\boldsymbol{\mu},3},
\end{equation*}
    where we choose to denote 
\begin{align*}
    \langle f,g \rangle_{{\boldsymbol{\mu},1}} :&= \lim _{r \rightarrow 1^{-}} \frac{1}{2 \pi} \iint_{\mathbb{D} \times \mathbb T} \partial_1 f(z_1, re^{it}) ~ \overline{\partial_1 g(z_1, re^{it})}~  P_{\mu_1}(z_1) dA(z_1) dt  \\
    \langle f,g \rangle_{\boldsymbol{\mu},2}:&= \lim _{r \rightarrow 1^{-}} \frac{1}{2 \pi} \iint_{\mathbb{T} \times \mathbb D} \partial_2 f(re^{it},z_2)~ \overline{\partial_2 g(re^{it},z_2)} ~  P_{\mu_2}(z_2) dA(z_2) dt \\
    \langle f,g \rangle_{\boldsymbol{\mu},3}:&=\iint_{\mathbb D^2} \partial_1\partial_2 f(\boldsymbol{z}) ~\overline{\partial_1 \partial_2 g(\boldsymbol{z})} ~\boldsymbol{P}_{\boldsymbol{\mu}}(\boldsymbol{z}) d\boldsymbol{A}(\boldsymbol{z}).
\end{align*}

We start with the following important lemma.
\begin{lemma}\label{Dirichlet contained in Hardy}
    For any $\boldsymbol{\mu} \in PM_{+}(\mathbb T^2), ~~\mathcal{D}^{2}(\boldsymbol{\mu}) \subseteq H^2(\mathbb D^2).$
\end{lemma}    
\begin{proof}
    Let $f(z_1,z_2) = \sum_{m,n=0}^\infty a_{m,n} z_1^{m} z_2^{n} \in \mathcal{D}^{2}(\boldsymbol{\mu}).$ Using monotone convergence theorem, we get 
\begin{align*}
    D_{\boldsymbol{\mu},1}^{(2)}(f)
     \geqslant \lim_{R \to 1^-}\lim_{r \to 1^-}  \frac{1}{2\pi} \displaystyle\iint_{R\mathbb{D} \times \mathbb{T}} \left|\partial_1 f(z_1, re^{it})\right|^{2}   P_{\mu_1}(z_1) dt dA(z_1).
\end{align*}
    Since $P_{\mu_1}(z_1) \geqslant \frac{\mu_1(\mathbb T)}{8\pi} (1-|z_1|^2)$ (see \cite[pg 236]{Rudin-R and C}), we get that,
\begin{align*}
    D_{\boldsymbol{\mu},1}^{(2)}(f) 
     &\geqslant \lim_{r \to 1^-} \lim_{R \to 1^-} \frac{1}{2\pi} \iint_{R\mathbb{D} \times \mathbb{T}}   \bigg|\sum_{m=1,n=0}^\infty m a_{m,n} z_1^{m-1} (re^{it})^{n}\bigg|^2 ~ \frac{\mu_1(\mathbb T)}{8\pi} (1-|z_1|^2) dt~ dA(z_1)\\
     &=  \lim_{r \to 1^-} \lim_{R \to 1^-} \frac{\mu_1(\mathbb T)}{4 \pi}\sum_{m = 1}^\infty \sum_{n=0}^\infty m^2 |a_{m,n}|^2 r^{2n} R^{2m}  \left \{\frac{1}{2m} -\frac{R^2}{2m+2}\right\} \\ 
     &=  \lim_{r \to 1^-} \frac{\mu_1(\mathbb T)}{8\pi}  \sum_{m = 1}^\infty \sum_{n=0}^\infty |a_{m,n}|^2 \frac{m}{m+1} r^{2n}.
\end{align*}
Since $\frac{m}{m+1} \geqslant \frac{1}{2},$ it follows that 
\begin{align}\label{lower bound for first integral}
    D_{\boldsymbol{\mu},1}^{(2)}(f) \geqslant  \lim_{r \to 1^-} \frac{\mu_1(\mathbb T)}{16 \pi} \sum_{m = 1}^\infty \sum_{n=0}^\infty |a_{m,n}|^2  r^{2n}.
\end{align}
    Similarly we also get that,  
\begin{equation}\label{lower bound for second integral}
     D_{\boldsymbol{\mu},2}^{(2)}(f) \geqslant \lim_{r \to 1^-} \frac{\mu_2(\mathbb T)}{16 \pi}  \sum_{m = 0}^\infty \sum_{n=1}^\infty |a_{m,n}|^2  r^{2m}.
\end{equation}
    Since $f \in \mathcal{D}^{2}(\boldsymbol{\mu})$, the left hand sides and consequently the right hand sides of \eqref{lower bound for first integral} and \eqref{lower bound for second integral} are finite.
    Now, it follows  from the Dominated Convergence Theorem that $f$ belongs to $H^2(\mathbb D^2).$
\end{proof}
    In the following theorem, we note that it is, in fact, a reproducing kernel Hilbert space.
\begin{theorem}\label{Dirichlet spaces are RKHS}
For any $\boldsymbol{\mu}\in \mathcal{PM}_{+}(\mathbb T^2),$ the Dirichlet-type space
    $(\mathcal{D}^{2}(\boldsymbol{\mu}), \|\cdot\|_{\boldsymbol{\mu}})$ is a reproducing kernel Hilbert space.
\end{theorem}
\begin{proof}
    From Lemma \ref{Dirichlet contained in Hardy}, we have that $\mathcal{D}^{2}(\boldsymbol{\mu}) \subseteq H^2(\mathbb D^2).$ This implies that $(\mathcal{D}^{2}(\boldsymbol{\mu}), \|\cdot\|_{\boldsymbol{\mu}})$ is contained contractively in $H^2(\mathbb D^2).$ Hence $\mathcal{D}^{2}(\boldsymbol{\mu})$ is a reproducing kernel Hilbert space provided $\mathcal{D}^{2}(\boldsymbol{\mu})$ is a Hilbert space with respect to $\|\cdot\|_{\boldsymbol{\mu}}.$ 
    To prove the completeness of the space $(\mathcal{D}^{2}(\boldsymbol{\mu}), \|\cdot\|_{\boldsymbol{\mu}})$, 
    let $(f_n)$ be a Cauchy sequence in $\mathcal{D}^{2}(\boldsymbol{\mu}).$ This in turn implies that $(f_n)$ is a Cauchy sequence in $H^2(\mathbb D^2).$ Since $H^2(\mathbb D^2)$ is a Hilbert space, there exists an $f \in H^2(\mathbb D^2)$ such that $f_n \rightarrow f$ as $n \rightarrow \infty$ in $H^2(\mathbb D^2)$. Fix an $\epsilon > 0.$ There exists an $N_1 \in \mathbb N$ such that 
    $\|f_n -f\|_{H^2(\mathbb D^2)} < \frac{\epsilon}{2}$  for all $n \geqslant N_1.$  Note that 
\begin{align*}
     & D^{(2)}_{\boldsymbol{\mu}}(f_n -  f) = D_{\boldsymbol{\mu}, 1}^{(2)}(f_n-f)+ D_{\boldsymbol{\mu}, 2}^{(2)}(f_n-f)+ D_{\boldsymbol{\mu}, 3}^{(2)}(f_n-f) \\  
     &= \lim_{r \to 1^-} \lim_{r_1 \to 1^-} \iint_{ r_1\mathbb D \times \mathbb T}  \left|{\partial_1{(f_n -f)}}(z_1, re^{it})\right|^{2} P_{\mu_1}(z_1)~dt~ dA(z_1) + \\ 
     & \lim_{r \to 1^-} \lim_{r_1 \to 1^-} \iint_{\mathbb T \times r_1 \mathbb D}  \left|{\partial_2 {(f_n -f)}} (re^{it}, z_2)\right|^{2} P_{\mu_2}(z_2)~ dA(z_1)~dt + \\
     &\lim_{r_1 \to 1^-}\iint_{ (r_1\mathbb D)^2} \left | {\partial_1 \partial_2 (f_n-f)}(\boldsymbol{z}) \right |^2 \boldsymbol{P}_{\boldsymbol{\mu}}(\boldsymbol{z}) d\boldsymbol{A}(\boldsymbol{z}).
\end{align*}
    Since, for each $n\in\mathbb N,$ $f_n$ is an analytic map, therefore $f_n, {\partial_1 f_n},$ ${\partial_2 f_n},$ and ${\partial_1 \partial_2 f_n}$ will converge uniformly on any compact subset of bidisc $\mathbb D^2$ respectively to $f, {\partial_1 f},$ ${\partial_2 f},$ and ${\partial_1 \partial_2 f}$. 
    For each $n\in\mathbb N,$ let $g_n(z_1, re^{it}) := \partial_1(f_n-f)(z_1,re^{it}).$ Then we have, 
\begin{align*}
      \iint_{r_1 \mathbb D \times \mathbb T} \displaystyle\lim_{n \to \infty} \left |g_n(z_1,re^{it}) \right|^2 P_{\mu_1}(z_1)~dt~ dA(z_1) =0. 
\end{align*}
     Therefore, by the Monotone Convergence theorem, there exists $N_2\in\mathbb N$ such that 
\begin{align*}
     \lim_{r \to 1^-}  \lim_{r_1 \to 1^-} \iint_{r_1 \mathbb D \times \mathbb T} \left |g_n(z_1,re^{it}) \right|^2 P_{\mu_1}(z_1)~dt~ dA(z_1) < \frac{\epsilon}{6}, \quad n \geqslant N_2.  
\end{align*}
Therefore, we get 
    $$\lim_{r \to 1^-} \iint_{\mathbb D \times \mathbb T}  \left|{\partial_1{(f_n -f)}} (z_1, re^{it})\right|^2 P_{\mu_1}(z_1)~dt~ dA(z_1) < \frac{\epsilon}{6}, \quad n \geqslant N_2.$$
    Similarly, we get that for some $N_3, N_4\in\mathbb N,$
\begin{align*}
       \lim_{r \to 1^-} \iint_{\mathbb T \times \mathbb D}  \left|{\partial_2{(f_n -f)}}(re^{it}, z_2)\right|^{2} P_{\mu_2}(z_2)~ dA(z_1) dt~< \frac{\epsilon}{6}, \quad n\geqslant N_3,
\end{align*} 
and 
\begin{align*}
   \lim_{r \rightarrow 1^-}\iint_{ (r_1\mathbb D)^2}  \left | {\partial_1\partial_2 (f_n-f)}(z_1,z_2) \right |^2 \boldsymbol{P}_{\boldsymbol{\mu}}(\boldsymbol{z}) d\boldsymbol{A}(\boldsymbol{z})  < \frac{\epsilon}{6}, \quad n\geqslant N_4.
\end{align*}
    By choosing $N=\max\{N_1,N_2,N_3,N_4\},$ we observe that $D^{(2)}_{\boldsymbol{\mu}}(f_n -  f) < \frac{\epsilon}{2},$ for all $n \geqslant N.$ 
   Since $(f_n)$ is Cauchy in $\mathcal{D}^{2}(\boldsymbol{\mu}),$ it is a bounded sequence. That is, there exists a $K > 0$ such  that, 
     $D^{(2)}_{\boldsymbol{\mu}}(f{_n}) < K$ for all $n\in\mathbb N.$ 
     Using the inequality $\|f\|_{\boldsymbol{\mu}} \leqslant \|(f-f_N)\|_{\boldsymbol{\mu}}+\|f_N\|_{\boldsymbol{\mu}},$ we conclude that $f\in\mathcal D^2(\boldsymbol \mu).$
     Since, for all $n\geqslant N,$ we have $D^{(2)}_{\boldsymbol{\mu}}(f_n -  f) < \frac{\epsilon}{2}$ and $\|f_n-f\|_{H^2(\mathbb D^2)}<\frac{\epsilon}{2},$ it follows that 
     Using this  and the fact that $f_n \to f$ in $\mathcal{D}^{2}(\boldsymbol{\mu})$ it follows that, $(f_n)$ converges to $f$ in  $\mathcal{D}^{2}(\boldsymbol{\mu}).$ This completes the proof of the theorem.  
\end{proof}
    When $\boldsymbol{\mu}$ is the Lebesgue measure on $\mathbb T^2$, then we denote $\mathcal{D}^{2}(\boldsymbol{\mu})$ by $\mathcal{D}^{2}$ and call it classical Dirichlet space over bidisc. We note that  
\begin{equation*}
    \mathcal{D}^{2} = \bigg\{   
    f(z_1,z_2)=\sum_{m, n =0}^{\infty}  a_{m,n}z_1^{m}z_2^{n} :  \sum_{m, n =0}^{\infty} (m+1)(n+1) |a_{m,n}|^{2} < \infty\bigg\},  
\end{equation*} 
    and consequently norm of any element $f\in\mathcal D^{(2)}$ with power series representation $\sum_{m, n =0}^{\infty}  a_{m,n}z_1^{m}z_2^{n}$, is given by 
    \[\|f\|_{\mathcal{D}^{2}}^2=\sum_{m, n =0}^{\infty} (m+1)(n+1) |a_{m,n}|^{2}.\]
    The associated  kernel function $K_{\mathcal{D}^{2}}$ at $(z_1,z_2),(w_{1},w_{2}) \in \mathbb{D}^{2}$ therefore turns out to be
\begin{equation*}
    K_{\mathcal{D}^{2}}\big((z_1,z_2),(w_1,w_2)\big)= \sum_{m,n=0}^\infty \frac{z_1^{m}\overline{w}_{1}^{m}~~ z_2^{n}\overline{w}_{2}^{n}}{(m+1)(n+1)}=K_{\mathcal D}(z_1,w_1) K_{\mathcal D}(z_2,w_2), 
\end{equation*}  
    where $K_{\mathcal D}$ is the reproducing kernel for the classical Dirichlet space $\mathcal D$.
    For any $f\in\mathcal O(\mathbb D^2)$ and fixed $(w_1,w_2)\in\mathbb D^2,$ we define $f_{w_1}(z):=f(w_1,z)$ and $f_{w_2}(z):=f(z,w_2)$ for any $z\in\mathbb D.$ We may refer $f_{w_1}$ and $f_{w_2}$ as $w_1$-slice and $w_2$-slice of $f$ respectively. The next proposition tells us that for any $f\in\mathcal D^2,$ its $w_1$-slice as well as $w_2$-slice belong to  Dirichlet space $\mathcal D.$
\begin{prop}\label{slice theorem for classical Dirichlet space}
     Suppose $(w_1,w_2)\in\mathbb D^2$ and $f$ is a function in $\mathcal{D}^{2}.$ Then $f_{w_1}$ and $f_{w_2}$ are in the classical Dirichlet space $\mathcal{D}.$ 
\end{prop}
\begin{proof}
    Let  $f(z_1,z_2)$ be represented by the power series $\sum_{m,n=0}^\infty a_{m,n} z_1^{m}z_2^{n}$. Fix $(w_1,w_2) \in \mathbb D^2.$ For each $m \in \mathbb Z_{+}$, define $b_{m} = \sum_{n=0}^\infty a_{m,n} w_2^{n}.$ In these notations, we have
\begin{align*}
    f_{w_2}(z) = \sum_{m=0}^\infty b_{m} z^{m}.
\end{align*}
    Note that
\begin{align*}
    \sum_{m=0}^\infty (m+1) |b_{m}|^{2} 
    \leqslant \frac{1}{1- |w_2|^{2}} \sum_{m,n=0}^\infty (m+1) |a_{m,n}|^{2} 
     \leqslant \frac{1}{1- |w_2|^{2}} \sum_{m,n=0}^\infty (m+1)(n+1) |a_{m,n}|^{2}.
\end{align*}
      This shows that $f_{w_2} \in \mathcal{D}.$
      Similarly, one can show that $f_{w_1} \in \mathcal{D}$. 
    \end{proof} 
    Proposition \ref{slice theorem for classical Dirichlet space} is true for general $\boldsymbol{\mu}\in\mathcal P\mathcal M_{+}(\mathbb T^2)$ when $(w_1,w_2)=0.$ We shall need this in what follows, e.g. in Lemma \ref{third integral estimate for z_1f}. Moreover, it says that $f(0,z_2)$ and $f(z_1,0),$ when treated as functions on bidisc, belong to Dirichlet-type space $\mathcal{D}^{2}(\boldsymbol{\mu}).$ 
    
    Throughout the paper, we will consistently adhere to the following notations:  for $m,p \in \mathbb Z_{\geqslant 0}$, define 
    $m \wedge p:=\min\{m,p\}$ and $m \vee p:=\max\{m,p\}$. Also, we choose $S$ to denote the open square $(0,2\pi) \times (0,2\pi)$. 
\begin{lemma}\label{semi-norm of slice functions}
    For $f \in \mathcal{D}^{2}(\boldsymbol{\mu})$, the functions $f(0,z_2)$ and $f(z_1,0) \in \mathcal{D}^{2}(\boldsymbol{\mu})$. Moreover, $D_{\boldsymbol{\mu}, j}^{(2)}(f)\geqslant D_{\boldsymbol{\mu}, j}^{(2)}(f(0,z_2))$ and $D_{\boldsymbol{\mu}, j}^{(2)}(f)\geqslant D_{\boldsymbol{\mu}, j}^{(2)}(f(z_1,0)),$ for $j=1,2,3.$
\end{lemma}
\begin{proof}
    Consider $f(z_1,z_2) = \sum_{m,n=0}^\infty a_{m,n} z_1^m z_2^n$. If we let $z_1=0$, then the resulting function $f_2=f(0,z_2)$ will be $f(0,z_2) = \sum_{n=0}^\infty a_{0,n} z_2^n$. So it is enough to check $D^{(2)}_{\boldsymbol{\mu},2}(f) < \infty$. The same computations as above show that
\begin{eqnarray*}
    D_{\boldsymbol{\mu}, 2}^{(2)}(f_2,R)
    =\frac{1}{2\pi}  \sum_{n,q = 1}^\infty (n \wedge q) a_{0,n}\overline{a_{0,q}}~ \widehat\mu_2(q - n)~R^{2(n\vee q)}.
\end{eqnarray*} 
Thus from \eqref{formula for second integral}, it can be seen that 
\begin{eqnarray*}
    D_{\boldsymbol{\mu}, 2}^{(2)}(f,R)&=&\lim_{r \rightarrow 1^-}\frac{1}{2\pi} \sum_{m=0}^\infty \sum_{n,q = 1}^\infty (n \wedge q) a_{m,n}\overline{a_{m,q}}~ \widehat\mu_2(q - n)~R^{2(n\vee q)} r^{2m} \\
    &=&  D_{\boldsymbol{\mu}, 2}^{(2)}(f_2,R) +  D_{\boldsymbol{\mu}, 2}^{(2)}\Big(\sum_{m=1}^\infty \sum_{n=0}^\infty a_{m,n} z_1^m z_2^n,R\Big).
\end{eqnarray*}
Taking limit $R$ tending to $1,$ this implies that $D_{\boldsymbol{\mu}, 2}^{(2)}(f)\geqslant D_{\boldsymbol{\mu}, 2}^{(2)}(f_2)$. As $D_{\boldsymbol{\mu}, 1}^{(2)}(f_2)=0=D_{\boldsymbol{\mu}, 3}^{(2)}(f_2),$ the proof of the lemma is completed.
\end{proof}

    Now, we proceed towards proving that the functions $z_1$ and $z_2$ are multipliers of $\mathcal{D}^{2}(\boldsymbol{\mu}).$ To this end, few notations and a number of lemmas are in order.
    
    For any $f\in\mathcal O(\mathbb D^2)$ with power series representation $\sum_{m,n=0}^{\infty}a_{m,n}z_1^{m} z_2^{n},$ define 
    $$(T_{k_1,k_2}f)(z_1,z_2) := \sum_{m=k_1}^\infty \sum_{n=k_2}^\infty a_{m,n} z_1^{m} z_2^{n}, \quad k_1,k_2\in\mathbb Z_{\geqslant 0}.$$
    Following Richter's paper \cite[page 330]{Richter1}, note that for any $R\in (0,1),$ and $t\in [0,2\pi],$    \begin{equation}\label{inner product-monomials}
    \frac{1}{2\pi} \int_0^{2\pi} e^{i(n-m)s} P(R e^{it},e^{is}) ds = R^{|n-m|} e^{i(n-m)t}, \quad m,n\in\mathbb Z_{\geqslant 0}.
    \end{equation}
    In what follows, we shall be using the fact that for any $R\in [0,1],$ $t,s\in [0,2\pi],$ one gets  $P(R e^{it},e^{is})=P(Re^{is},e^{it})$.
    
Let $\nu$ be a finite positive Borel measure  on the unit circle $\mathbb T$ and $n\in\mathbb Z,$ the $n^{\it {th}}$-Fourier coefficient of $\nu$ is defined by, 
\begin{align*}
    \widehat{\nu}(n) := \int_{0}^{2\pi} e^{-int} d\nu(t).
\end{align*}
At many occasions, we shall be dealing with quantities of the form
    \begin{eqnarray}\label{for short hand notation}
        \sum_{m,p=1}^\infty \sum_{n=0}^\infty (m \wedge p) a_{m,n} \overline{a}_{p,n} \widehat\nu(p-m) \mbox{ and }
        \sum_{n,q=1}^\infty \sum_{m=0}^\infty (n \wedge q) a_{m,n} \overline{a}_{m,q} \widehat\nu(q-n).
    \end{eqnarray}
    We shall find it convenient to adopt the following notation to denote the quantities in \eqref{for short hand notation}: For any measure $\nu\in\mathcal M_+(\mathbb T),$ define the matrix $M_{\nu}:=\big(\!\big((j \wedge k)  \widehat\nu(k - j)\big)\!\big).$ This is a matrix of size infinite and we shall denote it's $(j,k)$-th entry by $M_{\nu}(j,k).$ In these notations, the quantities in \eqref{for short hand notation} can be denoted as 
    \begin{eqnarray*}
        \sum_{n=0}^\infty \langle M_{\nu} \boldsymbol{a}_{\cdot n}, \boldsymbol{a}_{\cdot n}\rangle \mbox{ and }
        \sum_{m=0}^\infty \langle M_{\nu} \boldsymbol{a}_{m\cdot}, \boldsymbol{a}_{m\cdot}\rangle
    \end{eqnarray*}
     respectively, where $\boldsymbol{a}_{\cdot n}$ denotes the column $(a_{1n},a_{2n},\ldots)^{\tt T}$ and $\boldsymbol{a}_{m\cdot}$ denotes the column $(a_{m1},a_{m2},\ldots)^{\tt T}$. 
     Note that for any $\nu\in\mathcal M_+(\mathbb T),$ the matrix $M_{\nu}$ is formally positive definite. We shall use this fact in what follows, e.g. in Lemma \ref{R-dilation is contraction}. 
    
    We have the following lemmas which are analogous to the results of Richter in \cite{Richter1}. 
\begin{lemma}\label{first integral and tail}
    Let $0 < R < 1.$ If $f$ is analytic on $\mathbb D^2$, then
\begin{align*}
    D_{\boldsymbol{\mu}, 1}^{(2)}(f,R) = \lim_{r \to 1^-} \sum_{k_1=1}^\infty \frac{1}{(2\pi)^2} \iint_S \left |(T_{k_1,0} f)(Re^{is},re^{it}) \right|^2 P_{\mu_1}(Re^{is}) dsdt.
\end{align*}  
\end{lemma}
\begin{proof}
    For any $f(z_1,z_2)\in\mathcal O(\mathbb D^2)$ with power series representation $\sum_{m,n=0}^{\infty}a_{m,n}z_1^{m} z_2^{n}$ and $k_1\in\mathbb Z_{\geqslant 0}.$ Since $\int_0^{2\pi} e^{i(m-n)t} dt=\delta_{m,n},$ it follows that  
\begin{align*}
     D_{\boldsymbol{\mu}, 1}^{(2)}(f,R) &=\lim_{r \to 1^-} \frac{1}{2\pi} \iint_{R\mathbb{D} \times \mathbb{T}}  \left|\partial_1{f}(z_1, re^{it})\right|^{2} P_{\mu_1}(z_1) ~dt~ dA(z_1)\\
     &=  \lim_{r \to 1^-} \frac{1}{2\pi} \iint_{R\mathbb{D} \times \mathbb{T}}  \left|\sum_{m=1}^\infty \sum_{n=0}^\infty m a_{m,n} z_1^{m-1} (re^{it})^{n} \right|^2  P_{\mu_1}(z_1) ~dt~ dA(z_1)\\
    &=  \lim_{r \rightarrow 1^-}\frac{1}{2\pi} \frac{1}{\pi} \sum_{m,n,p=0}^\infty  m p a_{m,n}\overline{a_{p,n}} \frac{R^{2(m \vee p)}}{2(m \vee p)}   \widehat\mu_1(p - m)~ 2\pi~ r^{2n}.
\end{align*}
    Therefore we get that 
    \begin{equation}\label{LHS-first-R-integral}
    D_{\boldsymbol{\mu}, 1}^{(2)}(f,R) =\lim_{r \to 1^-} \frac{1}{2\pi} \sum_{m,p=1}^\infty \sum_{n=0}^\infty M_{\mu_1}(m,p) a_{m,n} \overline{a_{p,n}} r^{2n} R^{2(m \vee p)}.
\end{equation}
On the other hand, using \eqref{inner product-monomials}, note that
\begin{eqnarray}\label{RHS-first-R-integral}
    &&\lim_{r \to 1^-} \sum_{k_1=1}^\infty \frac{1}{(2\pi)^2} \iint_S \left |(T_{k_1,0} f)(Re^{is_1},re^{it}) \right|^2 P_{\mu_1}(Re^{is_1}) ds_1 dt\nonumber\\ 
   &=& \lim_{r \to 1^-} \sum_{k_1=1}^\infty \frac{1}{(2\pi)^2} \iint_S \left | \sum_{m=k_1}^\infty \sum_{n=0}^\infty a_{m,n} (Re^{is_1})^{m} (re^{it})^{n} \right |^2 P_{\mu_1}(Re^{is_1}) ds_1 dt \nonumber\\  
   &=& \lim_{r \to 1^-} \sum_{k_1=1}^\infty \frac{1}{(2\pi)^2} \int_{0}^{2 \pi} \sum_{m,p=k_1}^\infty \sum_{n,q=0}^\infty a_{m,n}\overline{a_{p,q}} R^{2(m \vee p)} r^{n+q} e^{i(n-q)t} \widehat\mu_1(p-m) dt \nonumber \\ 
    &=& \lim_{r \to 1^-} \sum_{k_1=1}^\infty \frac{1}{2\pi} \sum_{m,p=k_1}^\infty \sum_{n=0}^\infty a_{m,n}\overline{a_{p,n}} R^{2(m \vee p)} r^{2n} \widehat\mu_1(p-m)\nonumber\\
    &=& \lim_{r \to 1^-} \frac{1}{2\pi} \sum_{m,p=1}^\infty \sum_{n=0}^\infty (m \wedge p) a_{m,n} \overline{a_{p,n}} r^{2n} R^{2(m \vee p)} \widehat\mu_1(p-m).
    \end{eqnarray}
Equations \eqref{LHS-first-R-integral} and \eqref{RHS-first-R-integral} put together completes the proof of lemma.
\end{proof}
    We note down the following formula which is analogous to \eqref{LHS-first-R-integral}. In fact, we also have Lemma \ref{formula for second integral interms of tail of the taylor}. We omit the proof of the formula as well as Lemma \ref{formula for second integral interms of tail of the taylor}, since these are analogous to that of Lemma \ref{first integral and tail}.
\begin{eqnarray}\label{formula for second integral}
    D_{\boldsymbol{\mu}, 2}^{(2)}(f,R)=\lim_{r \rightarrow 1^-}\frac{1}{2\pi} \sum_{m=0}^\infty \sum_{n,q = 1}^\infty M_{\mu_2}(n,q) a_{m,n}\overline{a_{m,q}}~R^{2(n\vee q)} r^{2m}.
\end{eqnarray}
\begin{lemma}\label{formula for second integral interms of tail of the taylor}
    Let $0 < R < 1.$ If $f$ is analytic on $\mathbb D^2$, then
    \begin{align*}
     D_{\boldsymbol{\mu}, 2}^{(2)}(f,R)  = \lim_{r \to 1^-} \sum_{k_2=1}^\infty \frac{1}{(2\pi)^2} \iint_S \left |(T_{0,k_2} f)(re^{it}, Re^{is}) \right|^2 P_{\mu_2}(Re^{is}) ds dt.
\end{align*} 
\end{lemma}
    The following lemma gives a relationship between $D_{\boldsymbol{\mu}, 1}^{(2)}(z_1f,R)$ and $D_{\boldsymbol{\mu}, 1}^{(2)}(f,R).$ In particular, it proves that, for any $0<R<1,$ $D_{\boldsymbol{\mu}, 1}^{(2)}(z_1f,R)\geqslant R^2D_{\boldsymbol{\mu}, 1}^{(2)}(f,R).$ 
\begin{lemma}\label{first integral- difference between z_1f and f}
    Let $0 < R <1$. If $f$ is analytic on $\mathbb D^2,$ then  
\begin{align}\label{first-difference-formula-with-R}
    D_{\boldsymbol{\mu}, 1}^{(2)}(z_1f,R)-R^2 D_{\boldsymbol{\mu}, 1}^{(2)}(f,R) = 
      \lim_{r \to 1^-} \frac{R^2}{(2\pi)^2} \iint_S \left| f(Re^{is}, re^{it}) \right |^2 P_{\mu_1}(Re^{is}) ds dt.
\end{align}    
\end{lemma}
\begin{proof}
   Applying \eqref{LHS-first-R-integral} to $z_1f,$ we get 
   \begin{eqnarray}\label{first-integral-z_1f}
   D_{\boldsymbol{\mu}, 1}^{(2)}(z_1f,R)= R^2 ~\lim_{r \rightarrow 1^-}\frac{1}{2\pi} \sum_{m,n,p=0}^\infty M_{\mu_1}(m+1,p+1) a_{m,n}\overline {a_{p,n}} ~  R^{2(m\vee p)}~r^{2n}.
   \end{eqnarray}
On the other hand, using \eqref{inner product-monomials}, we observe that
\begin{eqnarray}\label{difference-z_1f_f_first-integral}
R.H.S. \mbox{ of } \eqref{first-difference-formula-with-R}= \lim_{r \to 1^-} \frac{R^2}{2\pi} \sum_{m,n,p=0}^\infty  a_{m,n} \overline{a_{p,n}} R^{2(m\vee p)} r^{2n}\widehat\mu_1(p-m).
\end{eqnarray}
  Equations \eqref{LHS-first-R-integral}, \eqref{first-integral-z_1f}, and \eqref{difference-z_1f_f_first-integral} put together completes the proof of the lemma.
\end{proof}
{For any $f\in \mathcal{D}^2(\boldsymbol\mu),$ taking limit $R$ tends to $1$ in \eqref{first-difference-formula-with-R}, one gets 
\begin{align}\label{first-difference-formula}
    D_{\boldsymbol{\mu}, 1}^{(2)}(z_1f)- D_{\boldsymbol{\mu}, 1}^{(2)}(f) = 
      \lim_{r \to 1^-} \frac{1}{(2\pi)^2} \iint_S \left| \tilde f(e^{is}, re^{it}) \right |^2 P_{\mu_1}(e^{is}) ds dt,
\end{align} 
where $\tilde f$ denotes the boundary value function corresponding to $f.$} 

\begin{lemma}\label{intermidate estimate for first integral-difference z_1f and f}
    There exists a constant $C = C(\boldsymbol\mu)$ such that for every $f \in \mathcal{D}^2(\boldsymbol\mu)$ and $0 < R < 1,$
\begin{align*}
    \lim_{r \to 1^-} \frac{1}{(2\pi)^2} \iint_S \left|f(Re^{is}, re^{it}) \right|^2 P_{\mu_1}(Re^{is}) ds dt \leqslant C \|f\|^2_{\boldsymbol\mu}
\end{align*}
\end{lemma}
\begin{proof}
    Let $Q:=\lim_{r \to 1^-} \frac{1}{(2\pi)^2} \iint_S \left|f(Re^{is}, re^{it}) \right|^2 P_{\mu_1}(Re^{is}) ds dt.$
    Note that 
    $$f(Re^{is},re^{it}) = f(0,0)+(T_{1,0}f)(Re^{is},re^{it})+\sum_{n=1}^\infty a_{0,n}(re^{it})^{n}.$$ 
    Therefore Lemmas \ref{first integral and tail} and   \ref{Dirichlet contained in Hardy} imply that
\begin{align*}
    |f(Re^{is},re^{it})|^2 &\leqslant 3\bigg(|f(0,0)|^2+|(T_{1,0}f)(Re^{is_1},re^{it})|^2 + \Big|\sum_{n=1}^\infty a_{0,n}(re^{it})^{n}\Big|^2\bigg)\\ 
    & \leqslant 3\bigg(\|f\|^2_{H^2(\mathbb D^2)} + \sum_{k_1=1}^\infty |(T_{k_1,0} f)(Re^{is_1},re^{it})|^2 + \Big|\sum_{n=1}^\infty a_{0,n}(re^{it})^{n}\Big|^2\bigg).
\end{align*}
    Using Lemma \ref{first integral and tail}, we deduce that
\begin{align*}
    Q \leqslant \frac{3}{2\pi}\|f\|^2_{H^2(\mathbb D^2)}\widehat{\mu_1}(0) + 3 D_{\boldsymbol{\mu}, 1}^{(2)}(f,R) + \frac{3}{2\pi} \widehat{\mu_1}(0) \lim_{r \to 1^-}\sum_{m=0}^\infty \sum_{n=0}^\infty |a_{m,n}|^2 r^{2n}.
\end{align*} 
   Choosing $C = \max\big\{\frac{3\widehat{\mu_1}(0)}{\pi} , 3\big\}$ completes the proof of the lemma.
\end{proof}
In this lemma, it is observed that the multiplication operator $M_{z_1}$ behaves like an isometry with respect to semi-norm $D_{\boldsymbol{\mu}, 2}^{(2)}(\cdot).$

\begin{lemma}\label{second integral- difference z_1f and f}
    If $f \in \mathcal{O}(\mathbb D^2)$ and $0 < R<1$, then $D_{\boldsymbol{\mu}, 2}^{(2)}(z_1f,R)  = D_{\boldsymbol{\mu}, 2}^{(2)}(f,R).$ 
\end{lemma}
\begin{proof}
Proof of this lemma easily follows from \eqref{formula for second integral}, hence we skip the details.
\end{proof}
In the next two lemmas, we study the semi-norm $D_{\boldsymbol{\mu}, 3}^{(2)}(\cdot).$ We observe that $D_{\boldsymbol{\mu}, 3}^{(2)}(z_1f,R)\geqslant D_{\boldsymbol{\mu}, 3}^{(2)}(f,R)$ for any $0<R<1.$ 
\begin{lemma}\label{third integral and tail of Taylor series}
    Let $0 < R <1$. If $f$ is analytic on $\mathbb D^2,$ then 
\begin{align*}
    D_{\boldsymbol{\mu}, 3}^{(2)}(f,R) =   \sum_{k_1=1}^\infty \sum_{k_2=1}^\infty \frac{1}{(2\pi)^2} \iint_S |(T_{k_1,k_2} f)(Re^{is_1},Re^{is_2})|^2 \boldsymbol{P}_{\boldsymbol{\mu}}(Re^{is_1},Re^{is_2}) ds_1 ds_2.
\end{align*}
\end{lemma}
\begin{proof}
Let $f\in\mathcal O(\mathbb D^2).$ 
Using uniform convergence of $\partial_1\partial_2 f$ on $R\mathbb D^2$ and a bit of computation, we get
\begin{align}\label{third integral expression}
     D_{\boldsymbol{\mu}, 3}^{(2)}(f,R)=\frac{1}{4\pi^2}  \sum_{m,n=1}^\infty \sum_{p,q=1}^\infty   M_{\mu_1}(m,p)M_{\mu_2}(n,q) a_{m,n} \overline{a_{p,q}} R^{2(m\vee p)} R^{2(n\vee q)}.
\end{align}
Now, using \eqref{inner product-monomials}, we get that
\begin{align*}
    & \sum_{k_1,k_2=1}^\infty \frac{1}{(2\pi)^2} \iint_S |(T_{k_1,k_2} f)(Re^{is_1},Re^{is_2})|^2 \boldsymbol{P}_{\boldsymbol{\mu}}(Re^{is_1},Re^{is_2}) ds_1 ds_2\\ 
     &=\sum_{k_1,k_2=1}^\infty \frac{1}{(2\pi)^2} \iint_S \Big|\sum_{m=k_1}^\infty \sum_{n=k_2}^\infty a_{m,n} (Re^{is_1})^{m} (Re^{is_2})^{n} \Big|^2  \boldsymbol{P}_{\boldsymbol{\mu}}(Re^{is_1},Re^{is_2}) ds_1 ds_2 \nonumber \\
    &=\sum_{k_1,k_2=1}^\infty\frac{1}{(2\pi)^2} 
    \iint_{S}\sum_{m,p=k_1}^\infty \sum_{n,q=k_2}^\infty a_{m,n} \overline{a_{p,q}} R^{2(m\vee p)} R^{2(n\vee q)} e^{i(m-p)\theta_1} e^{i(n-q)\theta_2}d\mu_1(\theta_1) d\mu_2(\theta_2) \nonumber \\
    &= \frac{1}{4\pi^2}  \sum_{m=1}^\infty \sum_{n=1}^\infty \sum_{p=1}^\infty \sum_{q=1}^\infty  (m\wedge p) (n\wedge q) a_{m,n} \overline{a_{p,q}} R^{2(m\vee p)} R^{2(n\vee q)} \widehat \mu_1(p - m) \widehat\mu_2(q-n). 
\end{align*}
   This completes the proof of the lemma.
\end{proof}

\begin{lemma}\label{third integral estimate for z_1f}
    Let $0 < R<1.$ If $f$ is in $\mathcal{D}^{2}(\boldsymbol{\mu})$, then there exists a constant $C=C(\boldsymbol{\mu})$ such that $D_{\boldsymbol{\mu}, 3}^{(2)}(z_1f,R) \leqslant C \|f\|^2_{\boldsymbol{\mu}}.$ 
\end{lemma}
\begin{proof}
Since $T_{k_1,k_2} f=T_{k_1+1,k_2} (z_1f)$ for all $k_1,k_2\in\mathbb Z_{\geqslant 0},$ it follows from Lemma \ref{third integral and tail of Taylor series} that 
\begin{align*}
    D_{\boldsymbol{\mu}, 3}^{(2)}(z_1f,R) =  \sum_{k_2=1}^\infty \frac{R^2}{(2\pi)^2} \iint_S  |(T_{0,k_2} f)(Re^{is_1},Re^{is_2})|^2 \boldsymbol{P}_{\boldsymbol{\mu}}(Re^{is_1},Re^{is_2}) ds_1 ds_2 +D_{\boldsymbol{\mu}, 3}^{(2)}(f,R)
\end{align*}
  Note that
\begin{align*}
    |(T_{0,k_2} f)(Re^{is_1},Re^{is_2})|^2 &= \Big|\sum_{m=0}^\infty \sum_{n=k_2}^\infty a_{m,n} (Re^{is_1})^{m} (Re^{is_2})^{n}\Big|^2 \\  
    &\leqslant 2\bigg(\Big|\sum_{n=k_2}^\infty a_{0,n}(Re^{is_2})^{n}\Big|^2 + |(T_{1,k_2}f)(Re^{is_1},Re^{is_2})|^2\bigg) \\
    & \leqslant 2\bigg(\Big|\sum_{n=k_2}^\infty a_{0,n}(Re^{is_2})^{n}\Big|^2 + \sum_{k_1=1}^\infty |(T_{k_1,k_2} f)(Re^{is_1},Re^{is_2})|^2\bigg).
\end{align*} 
    Therefore, using Lemma \ref{third integral and tail of Taylor series}, we get 
\begin{align}\label{4.7 computation}
     &\sum_{k_2=1}^\infty \frac{R^2}{(2\pi)^2} \iint_S |(T_{0,k_2} f)(Re^{is_1},Re^{is_2})|^2 \boldsymbol{P}_{\boldsymbol{\mu}}(Re^{is_1},Re^{is_2}) ds_1 ds_2 \nonumber \\
     &\leqslant \sum_{k_2=1}^\infty \frac{2R^2}{(2\pi)^2} \iint_S \Big|\sum_{n=k_2}^\infty a_{0,n}(Re^{is_2})^{n}\Big|^2 \boldsymbol{P}_{\boldsymbol{\mu}}(Re^{is_1},Re^{is_2}) ds_1 ds_2 +2R^2 D_{\boldsymbol{\mu}, 3}^{(2)}(f,R) 
\end{align}
    Since 
\begin{align*}
    \iint_{\mathbb T \times R\mathbb D} |\partial_2 f(0,z_2)|^2 & P_{\mu_2} (z_2)  dA(z_2) dt = \sum_{n,q=1}^{\infty} (m \wedge q) a_{0,n} ~\overline{a_{0,q}} \widehat{\mu_2}(q-n) R^{2(m \vee q)} \\
    & =  \frac{1}{\widehat{\mu_1}(0)} \sum_{k_2=1}^\infty \iint_S \Big|\sum_{n=k_2}^\infty a_{0,n}(Re^{is_2})^{n}\Big|^2 \boldsymbol{P}_{\boldsymbol{\mu}}(Re^{is_1},Re^{is_2}) ds_1 ds_2. 
\end{align*}
   The inequality in \eqref{4.7 computation} now gives
\begin{align*}
    &\sum_{k_2=1}^\infty \frac{R^2}{(2\pi)^2} \iint_S |(T_{0,k_2} f)(Re^{is_1},Re^{is_2})|^2 \boldsymbol{P}_{\boldsymbol{\mu}}(Re^{is_1},Re^{is_2}) ds_1 ds_2 \\
    & \leqslant \frac{2R^2}{2 \pi} \widehat{\mu_1}(0) D^{(2)}_{\boldsymbol{\mu},2} f(0,z_2) + 2R^2 D_{\boldsymbol{\mu}, 3}^{(2)}(f,R) 
\end{align*}
    By Lemma \ref{third integral and tail of Taylor series}, Lemma \ref{semi-norm of slice functions} and the above computation, we note that, 
\begin{align*}
    \iint_{R\mathbb D^2} \left | \partial_1 \partial_2 z_1f(\boldsymbol{z}) \right |^2 \boldsymbol{P}_{\boldsymbol{\mu}}(\boldsymbol{z}) d\boldsymbol{A}(\boldsymbol{z}) &\leqslant \frac{R^2}{\pi} \widehat{\mu_1}(0) D^{(2)}_{\boldsymbol{\mu},2} (f,R) + 2R^2 D_{\boldsymbol{\mu}, 3}^{(2)}(f,R) + D_{\boldsymbol{\mu}, 3}^{(2)}(f,R) \\
    & \leqslant C \| f\|^2_{\boldsymbol{\mu}}
\end{align*}
    where $C= \max \big\{ \frac{R^2}{\pi} \widehat{\mu_1}(0), 2R^2, 1\big\}$.
\end{proof}
The proof of above lemma, in the limit $R\to 1,$ implies that 
\begin{eqnarray}\label{third integral z1f is bigger than that of f}
    D_{\boldsymbol{\mu}, 3}^{(2)}(z_1f)\geqslant D_{\boldsymbol{\mu}, 3}^{(2)}(f).
\end{eqnarray}
\begin{theorem}\label{Gleason property at origin}
    Suppose $f\in\mathcal O(\mathbb D^2)$. Then the following statements are equivalent
    \begin{enumerate}
    \item[(i)] $f \in \mathcal{D}^{2}(\boldsymbol{\mu})$
    \item[(ii)] $z_1f \in \mathcal{D}^{2}(\boldsymbol{\mu})$
    \item[(iii)] $z_2f\in \mathcal{D}^{2}(\boldsymbol{\mu})$.
    \end{enumerate}
    Moreover, in this case,  $\|z_if\|_{\boldsymbol{\mu}}\geqslant  \|f\|_{\boldsymbol{\mu}}$ for $i=1,2.$ 
\end{theorem}
\begin{proof} It is enough to prove the equivalence of $(i)$ and $(ii)$ only.
Combining Lemma \ref{first integral- difference between z_1f and f}, Lemma \ref{intermidate estimate for first integral-difference z_1f and f}, Lemma \ref{second integral- difference z_1f and f}, and Lemma \ref{third integral estimate for z_1f} proves that $(i)$ implies $(ii)$. For the other way implication, we apply equation \eqref{first-difference-formula} with Lemma \ref{second integral- difference z_1f and f} and Equation \eqref{third integral z1f is bigger than that of f}.
\end{proof}
For an infinite matrix $X:(\!(x_{mp})\!),$ define $\sigma^*X$ (see \cite[page 151]{Shimorin}) to be the matrix whose $(m,p)$-th entry is given by $x_{m+1,p+1}.$
\begin{lemma}
       For any polynomial $p \in \mathcal{D}^{2}(\boldsymbol{\mu}),$ and $i=1,2$ we have the following relations;
\[\|z_i^2p\|^2_{\boldsymbol{\mu}} -2\|z_ip\|^2_{\boldsymbol{\mu}}+\|p\|^2_{\boldsymbol{\mu}} = 0.\]
\end{lemma}
\begin{proof}
    On computing the norm of $p,z_1p,z_1^2p$ for $p(z_1,z_2)= \sum_{m,n=0}^k a_{m,n} z_1^{m} z_2^{n}$ with respect to the defined norm on $\mathcal{D}^{2}(\boldsymbol{\mu})$,  we have the following.
    Note that 
    \[\|p\|_{H^2(\mathbb D^2)}^2 = \|z_1p\|_{H^2(\mathbb D^2)}^2 = \|z_1^2p\|_{H^2(\mathbb D^2)}^2=\sum_{m,n=0}^k  |a_{m,n}|^2.\] 
Therefore it is enough to check
    \[D_{\boldsymbol{\mu}}^{(2)}(z_i^2 p)- 2D_{\boldsymbol{\mu}}^{(2)}(z_i p)+ D_{\boldsymbol{\mu}}^{(2)}(p)=0.\]
   Now, from \eqref{LHS-first-R-integral}, \eqref{formula for second integral}, \eqref{third integral expression}, with $R=1,$ it follows that
\begin{align*}
    D_{\boldsymbol{\mu}, 1}^{(2)}(p)
    =  \frac{1}{2\pi} \sum_{n=0}^k \langle M_{\mu_1}\boldsymbol{a}_{\cdot n},\boldsymbol{a}_{\cdot n}\rangle \quad \mbox{and} \quad  D_{\boldsymbol{\mu}, 2}^{(2)}(p)
     =  \frac{1}{2\pi} \sum_{m=0}^k \langle M_{\mu_2}\boldsymbol{a}_{ m\cdot},\boldsymbol{a}_{m\cdot}\rangle.
\end{align*}
 In the same spirit, we also get the formulae:
\[D_{\boldsymbol{\mu}, 3}^{(2)}(p)
    = \frac{1}{4\pi^2} \langle (M_{\mu_1}\otimes M_{\mu_2}) \boldsymbol{a},  \boldsymbol{a}\rangle, \ D^{(2)}_{\boldsymbol{\mu},1}(z_1p) 
       =  \frac{1}{2\pi} \sum_{n=0}^k \langle \sigma^*M_{\mu_1}\boldsymbol{a}_{\cdot n},\boldsymbol{a}_{\cdot n}\rangle,\]
   \[D^{(2)}_{\boldsymbol{\mu},2}(z_1p) 
   =D^{(2)}_{\boldsymbol{\mu},2}(z_1^2p)
    = \frac{1}{2\pi} \sum_{m=0}^k \langle M_{\mu_2}\boldsymbol{a}_{m\cdot},\boldsymbol{a}_{m\cdot}\rangle,\]
   \[D^{(2)}_{\boldsymbol{\mu},3}(z_1p)      
   = 
\frac{1}{4\pi^2} \langle (\sigma^*M_{\mu_1}\otimes M_{\mu_2}) \boldsymbol{a}, \boldsymbol{a}\rangle, \ D^{(2)}_{\boldsymbol{\mu},1}(z_1^2p)
    = \frac{1}{2\pi} \sum_{n=0}^k \langle (\sigma^*)^2M_{\mu_1}\boldsymbol{a}_{\cdot n},\boldsymbol{a}_{\cdot n}\rangle,\]
   and 
\begin{eqnarray*}
D^{(2)}_{\boldsymbol{\mu},3}(z_1^2p)   
   = \frac{1}{4\pi^2} \langle ((\sigma^*)^2M_{\mu_1}\otimes M_{\mu_2}) \boldsymbol{a}, \boldsymbol{a}\rangle. 
\end{eqnarray*}
    The above computations put together prove that  
   $\|z_1^2 p\|_{\boldsymbol{\mu}}^2 - 2 \|z_1p\|_{\boldsymbol{\mu}}^2 + \|p\|_{\boldsymbol{\mu}}^2 = 0.$
    Similarly, we can prove that  $\|z_2^2 p\|_{\boldsymbol{\mu}}^2 - 2 \|z_2p\|_{\boldsymbol{\mu}}^2 + \|p\|_{\boldsymbol{\mu}}^2 = 0.$  
\end{proof}
    For any analytic function $f$ in $\mathbb D^2$ and $0 < r <1$, define $f_r(z_1,z_2) := f(rz_1, rz_2)$ for $(z_1,z_2)\in\overline{\mathbb D}$.
\begin{lemma}\label{R-dilation is contraction}
    For any $f \in \mathcal{D}^{2}(\boldsymbol{\mu})$,  the relation $D^{(2)}_{\boldsymbol{\mu}}(f_r) \leqslant D^{(2)}_{\boldsymbol{\mu}}(f)$ holds for any $0<r<1$.
\end{lemma}
\begin{proof}
    Let $f \in \mathcal{D}^{2}(\boldsymbol{\mu})$ and $0<r<1.$
   We shall prove the relations $D_{\boldsymbol{\mu}, i}^{(2)}(f_r) \leqslant D_{\boldsymbol{\mu}, i}^{(2)}(f)$, for $i=1,2,3$.  
    To this end, we claim that 
\begin{eqnarray}\label{first integral - contraction}
    D_{\boldsymbol{\mu}, 1}^{(2)}(f_r) \leqslant D_{\boldsymbol{\mu}, 1}^{(2)}(f).
\end{eqnarray}
    We will first prove \eqref{first integral - contraction} for polynomials and then for any given $f \in \mathcal{D}^{2}({\boldsymbol{\mu}})$. 
    Let $P(z_1,z_2)=\sum_{m,n=0}^k a_{m,n} z_1^m z_2^n$ be a polynomial of degree at most $2k.$
    Then, from \eqref{LHS-first-R-integral} with $R=1$, we have
\begin{align*}
    D^{(2)}_{\boldsymbol{\mu},1}(P) = 
    \frac{1}{2\pi} \sum_{n=0}^k \langle M_{\mu_1} \boldsymbol{a}_{\cdot n}, \boldsymbol{a}_{\cdot n}\rangle.
\end{align*} 
    Since $D^{(2)}_{\boldsymbol{\mu},1}(z_1^2 P) - 2D^{(2)}_{\boldsymbol{\mu},1}(z_1 P) +D^{(2)}_{\boldsymbol{\mu},1}(P)=0$, it then follows that $(\sigma^* -I)^2 M_{\mu_1}=0$. 
    Thus from \cite[Theorem 3.11]{Shimorin}, we get that the matrix $\big(\!\big(1-r^{m+p})M_{\mu_1}(m,p)\big)\!\big)_{m,p \geqslant 0}$ is positive semidefinite. Now we have
\begin{align*}
    \sum_{n=0}^{k} \langle \big((\boldsymbol{1}-\boldsymbol{r}^*\boldsymbol{r})\circ M_{\mu_1}\big) \boldsymbol{a}_{\cdot n}, \boldsymbol{a}_{\cdot n}\rangle \geqslant 0,
\end{align*} 
where $\boldsymbol{r}$ denotes the row vector $(1, r, r^2,\ldots, r^k),$ and  $\boldsymbol{1}$ denotes the rank one matrix with all its entries equal to $1.$
Here $X\circ Y$ denotes the Schur product of matrices $X$ and $Y.$
Therefore, it follows that
\begin{align*}
     D^{(2)}_{\boldsymbol{\mu},1}(P) - D^{(2)}_{\boldsymbol{\mu},1}(P_r) 
    & = \sum_{m,n,p=0}^{k} (1-r^{m+p+2n})M_{\mu_1}(m,p) a_{m,n} \overline{a_{p,n}} \\
    & = \sum_{n=0}^{k} \langle\big((\boldsymbol{1}-\boldsymbol{r}^*\boldsymbol{r})\circ M_{\mu_1}\big) \boldsymbol{a}_{\cdot n},\boldsymbol{a}_{\cdot n}\rangle  +
     \sum_{n=0}^{k} (1-r^{2n}) \langle\big(\boldsymbol{r}^*\boldsymbol{r}\circ M_{\mu_1}\big) \boldsymbol{a}_{\cdot n}, \boldsymbol{a}_{\cdot n}\rangle \\
     & \geqslant 0.
\end{align*}
    The last inequality follows since the matrix $(\!(r^{m+p})\!)_{m,p \geqslant 0}$ is positive definite. This proves the claim \eqref{first integral - contraction}. In a similar fashion, we can prove that $D^{(2)}_{\boldsymbol{\mu},2}(P_r) \leqslant D^{(2)}_{\boldsymbol{\mu},2}(P).$ 
    Note that from \eqref{third integral expression} with $R=1$, we have
\begin{align*}    D^{(2)}_{\boldsymbol{\mu},3}(P) = 
    \frac{1}{4\pi^2} \langle (M_{\mu_1}^{\tt T}\otimes M_{\mu_2}^{\tt T}) \boldsymbol{a}, \boldsymbol{a}\rangle,
\end{align*}
where $\boldsymbol{a}$ denotes the $2$-tensor $((a_{mn})).$ 
    As $(\sigma^* -I)^2 M_{\mu_1}=0$ and  $(\sigma^* -I)^2 M_{\mu_2}=0$, therefore $((\sigma^* -I)^2 M_{\mu_1}) \otimes M_{\mu_2}=0$ and $M_{\mu_1}\otimes((\sigma^* -I)^2 M_{\mu_2})=0$ hold true. 
    By \cite[Theorem 3.11]{Shimorin}, it turns out that for any $0\leqslant r < 1$ the matrices $\big(\!(\boldsymbol{1}-\boldsymbol{r}^*\boldsymbol{r})\circ M_{\mu_1})\!\big)\otimes M_{\mu_2}$
    and 
    $M_{\mu_1}\otimes \big(\!(\boldsymbol{1}-\boldsymbol{r}^*\boldsymbol{r})\circ M_{\mu_2})\!\big)$
    are positive semi-definite. 
    This implies that 
\begin{align}\label{Shimorin's result for matrix A}
    \sum_{m,n,p,q=0}^{k} (1-r^{m+p})M_{\mu_1}(m,p)M_{\mu_2}(n,q) a_{m,n} \overline{a_{p,q}} \geqslant 0 
\end{align} 
\begin{align}\label{Shimorin's result for matrix B}
\sum_{m,n,p,q=0}^{k} (1-r^{n+q})M_{\mu_1}(m,p)M_{\mu_2}(n,q) a_{m,n} \overline{a_{p,q}} \geqslant 0
\end{align}
Note that 
\begin{align*}
D^{(2)}_{\boldsymbol{\mu,3}}(P)-D^{(2)}_{\boldsymbol{\mu,3}}(P_r)&=\sum_{m,n,p,q=0}^{k} (1-r^{m+n+p+q}) M_{\mu_1}(m,p)M_{\mu_2}(n,q) a_{m,n} \overline{a_{p,q}} \nonumber \\ &=\sum_{m,n,p,q=0}^{k} (1-r^{m+p}+ r^{m+p}-r^{m+n+p+q}) M_{\mu_1}(m,p)M_{\mu_2}(n,q) a_{m,n} \overline{a_{p,q}}  \nonumber \\
    & = \sum_{m,n,p,q=0}^{k} (1-r^{m+p}) M_{\mu_1}(m,p)M_{\mu_2}(n,q) a_{m,n} \overline{a_{p,q}} ~+ \nonumber \\ & \quad \sum_{m,n,p,q=0}^{k}r^{m+p} (1-r^{n+q})M_{\mu_1}(m,p)M_{\mu_2}(n,q) a_{m,n} \overline{a_{p,q}} \nonumber \\
    &\geqslant 0.
\end{align*}
The last inequality follows from \eqref{Shimorin's result for matrix A}, \eqref{Shimorin's result for matrix B} and the fact that the matrix                  
$(\!(r^{m+p})\!)_{m,p=0}^{\infty}$ is positive       
    semi-definite. To see that the matrix    $(\!(r^{m+p})\!)_{m,p=0}^{\infty}$
is positive         
semi-definite, note that $(\!(r^{m+p})\!)=vv^T$, with $v$    
given by 
$v^T=\begin{bmatrix}
    1 & r & r^2 & r^3 & \cdots  
\end{bmatrix}.$ 
This proves the lemma if $f$ is a polynomial. Applying a standard uniform limit argument proves the lemma for any function $f$ which is holomorphic in a neighborhood of $\overline{\mathbb D}^2.$
    Now for $f \in \mathcal{D}^{2}(\boldsymbol{\mu})$ and $0<R<1$ we have $f_R \in \mathcal{O}(\overline{\mathbb D}^2).$ Therefore this readily implies that $D^{(2)}_{\boldsymbol{\mu}}(f_R)_r \leqslant D^{(2)}_{\boldsymbol{\mu}}(f_R),$ where $r$ being any number in $(0,1).$ Since $(f_R)_r = (f_r)_R,$ we get that $D^{(2)}_{\boldsymbol{\mu}}(f_r)_R \leqslant D^{(2)}_{\boldsymbol{\mu}}(f_R).$ This is true for any $0<R<1.$ Hence by taking limit $R \rightarrow 1,$ we get that $D^{(2)}_{\boldsymbol{\mu}}(f_r) \leqslant D^{(2)}_{\boldsymbol{\mu}}(f).$
\end{proof}
We skip the proof of the following theorem as it is standard and follows from \cite[Lemma 3.7]{Sameer et. al.} together with Fatou's theorem applied to $D_{\boldsymbol{\mu},3}^{(2)}(\cdot).$
\begin{lemma}\label{approximation by R-dilation}
    For any $f \in \mathcal{D}^{2}(\boldsymbol{\mu})$, $\|f_R-f\|_{\boldsymbol{\mu}} \rightarrow 0$ as $R \rightarrow 1^-.$
\end{lemma}
In the following theorem, we prove that the set of all polynomials is dense in $\mathcal{D}^{2}(\boldsymbol{\mu}).$
\begin{theorem}
    The set of all polynomials is dense in $\mathcal{D}^{2}(\boldsymbol{\mu}).$ 
\end{theorem}
\begin{proof}
    Let $f \in \mathcal{D}^{2}(\boldsymbol{\mu})$ and $\epsilon >0$. From Lemma \ref{approximation by R-dilation}, it follows that there exists an $0<R<1$ such that 
    $\|f_R-f\|_{\boldsymbol{\mu}} < \epsilon/2.$ 
    Let $f_R(z_1,z_2) = \sum_{m,n=0}^\infty b_{m,n} z_1^m z_2^n,$ where the coefficient $b_{m,n},$ for $m,n\geqslant 0,$ is a function of $R$. 
    Consider the $k$-th partial sum $S_kf_R(z_1,z_2)=\sum_{m,n=0}^k b_{m,n} z_1^m z_2^n$ of $f_R(z_1,z_2)$. Note that $f_R(z_1,z_2)$ is holomorphic in any neighbourhood of $\overline{\mathbb D}^2$. Therefore the sequences $(S_k f_R)_{k=0}^{\infty}$, $({\partial_1}(S_kf_R))_{k=0}^{\infty}$, $({\partial_2}(S_kf_R))_{k=0}^{\infty},$ and $({\partial_1 \partial_2}(S_kf_R))_{k=0}^{\infty}$ will converge uniformly on $\overline{\mathbb D}^2$ respectively to $f_R$, ${\partial_1} (f_R)$, ${\partial_2}(f_R),$  and ${\partial_1 \partial_2}(f_R)$. 
    Thus it follows that there exists $k\in\mathbb N$ such that $\|S_kf_R-f_R\|^2_{\mathcal{D}^2(\boldsymbol{\mu})}\leqslant {\epsilon/2}.$ 
    We thus conclude that 
    \begin{align*}
       \|f- S_kf_R \|_{\boldsymbol{\mu}} \leqslant \|f-f_R\|_{\boldsymbol{\mu}}+\|S_kf_R-f_R\|_{\boldsymbol{\mu}}
        < \frac{\epsilon}{2}+ \frac{\epsilon}{2}=\epsilon.
    \end{align*} 
    This completes the proof of the theorem.
\end{proof}
    The results due to Lemma 2.13 and Theorem 2.16 yield the following theorem. 
\begin{theorem}
    The multiplication by the co-ordinate functions $M_{z_1}$ and $M_{z_2}$ are $2$-isometries on ${\mathcal{D}^{2}(\boldsymbol{\mu})}$, i.e., for every $f \in \mathcal{D}^{2}(\boldsymbol{\mu}),$ and $i=1,2,$
\begin{center}
    $\|z_i^2 f\|_{\boldsymbol{\mu}}^2 - 2 \|z_if\|_{\boldsymbol{\mu}}^2 + \|f\|_{\boldsymbol{\mu}}^2 = 0.$   
\end{center}
\end{theorem}
Recall that for any $\nu\in\mathcal M_+(\mathbb T),$ one has 
\begin{align}\label{inner product in 1 variable}
    \langle z^m, z^n \rangle_{\mathcal{D}(\nu)}=
\begin{cases}
    \frac{m \wedge n}{2\pi} \widehat{\nu}(n-m) &\quad if m \neq n \\
    1+ \frac{m}{2\pi}  \widehat{\nu}(0) &\quad if m=n.
\end{cases}
\end{align}
In the following lemma, we get a formula for inner product of monomials in the space $\mathcal{D}^{2}(\boldsymbol{\mu})$, for any $\boldsymbol{\mu}\in\mathcal P\mathcal M_+(\mathbb T^2).$
\begin{lemma}\label{inner product- monomials - in term of single variable}
    For $m,n,p,q\in \mathbb Z_{\geqslant 0},$ $\langle z_1^m z_2^n, z_1^p z_2^q \rangle_{\boldsymbol{\mu}} = \langle z_1^m, z_1^p \rangle_{\mathcal{D}(\mu_1)} \langle z_2^n, z_2^q \rangle_{\mathcal{D}(\mu_2)}.$
\end{lemma}
\begin{proof} 
Let $m,n,p,q\in\mathbb Z_{\geqslant 0},$ $f(z_1,z_2):=z_1^m z_2^n$ and  $g(z_1,z_2):=z_1^p z_2^q.$
We divide the proof of the lemma into the following five cases: 
\begin{itemize}
    \item[\textbf{Case 1:}] When either $f$ or $g$ is the constant function $1,$ i.e. $m=n=0$ or $p=q=0.$
    \item[\textbf{Case 2:}] When $f$ and $g$ both are function of $z_1$ only (similarly of $z_2$ only).
    \item[\textbf{Case 3:}] When $f$ is a function of $z_1$ only and $g$ is a function of $z_2$ only.
    \item[\textbf{Case 4:}] When $f$ is a function of $z_1$ only and $g$ is a function of both $z_1$ and $z_2$.
    \item[\textbf{Case 5:}] When $f$ and $g$ both are functions of both variables $z_1$ and $z_2$, i.e. $m,n,p,q\in\mathbb N.$ 
    \end{itemize}
    
    \textbf{Case 1}: Suppose $f$ is the constant function $1$, then $$\langle 1, z_1^pz_2^q\rangle_{\boldsymbol{\mu}}=\langle 1, z_1^pz_2^q\rangle_{H^2(\mathbb D^2)}= \langle 1, z_1^p\rangle_{H^2(\mathbb D)} \langle 1, z_2^q\rangle_{H^2(\mathbb D)}=\langle 1, z_1^p\rangle_{\mathcal D(\mu_1)}\langle 1, z_2^q\rangle_{\mathcal D(\mu_2)}.$$
    
    \textbf{Case 2}: 
From \eqref{inner product-monomials}, 
we get that
    \begin{align*}
    \langle z_1^m, z_1^p\rangle_{\mathcal{D}^{2}_{\boldsymbol{\mu},1}}  &=   
    \lim _{r \rightarrow 1^{-}} \frac{1}{2 \pi} \iint_{\mathbb{D} \times \mathbb T} m z_1^{m-1}  p \bar{{z}_1}^{p-1} P_{\mu_1}(z_1) dA(z_1) dt  \\
    &= \lim _{r \rightarrow 1^{-}} \frac{1}{2 \pi} \frac{1}{\pi} \int_0^1 \int_0^{2\pi} \int_0^{2 \pi} m p~ r_1^{m+p-1}r_1^{|m - p|}e^{i(m - p)\theta_1}dr_1 dt d\mu_1 (\theta_1) \\ 
    &= \lim _{r \rightarrow 1^{-}} \frac{1}{2 \pi} \frac{1}{\pi}  mp \int_0^1  r_1^{m+p-1} r_1^{|m - p|} dr_1 \int_0^{2\pi} dt  \int_0^{2\pi} e^{i(m - p)\theta_1} d\mu_1(\theta _1) \\
    & =\frac{1}{2 \pi} (m \wedge p) \widehat{\mu}_1\left(p-m\right)
    \end{align*} 
    Whereas, note that $\langle z_1^m,z_1^p\rangle_{H^2\left(\mathbb{D}^2\right)}= \delta_{mp}$ and $ 0=\langle z_1^m, z_1^p\rangle_{\mathcal{D}^{2}_{\boldsymbol{\mu},2}}=\langle z_1^m, z_1^p\rangle_{\mathcal{D}^{2}_{\boldsymbol{\mu},3}}.$
    Therefore, we get 
    \[\langle z_1^m , z_1^p \rangle_{\boldsymbol{\mu}} =\delta_{mp}+\frac{1}{2 \pi} (m \wedge p) \widehat{\mu}_1\left(p-m\right).\]
    Now, using \eqref{inner product in 1 variable} completes the proof in this case.

\textbf{Case 3:} 
    Using the definitions,
    it follows that 
    $$\langle z_1^m,z_2^q\rangle_{H^2\left(\mathbb{D}^2\right)}=\langle z_1^m,z_2^q\rangle_{\mathcal{D}^{2}_{\boldsymbol{\mu},1}}=0=\langle z_1^m, z_2^q\rangle_{\mathcal{D}^{2}_{\boldsymbol{\mu},2}}=\langle z_1^m, z_2^q\rangle_{\mathcal{D}^{2}_{\boldsymbol{\mu},3}}.$$ 
    Thus, $z_1^m$ and $z_2^q$ are orthogonal in $\mathcal{D}^{(2)}(\boldsymbol{\mu})$. For any $m,q\in\mathbb N,$ it is trivial to verify that $\langle z_1^m, 1 \rangle_{\mathcal{D}(\mu_1)} \langle 1, z_2^q \rangle_{\mathcal{D}(\mu_2)}=0.$ 

    \textbf{Case 4:} 
From \eqref{inner product-monomials}, note that
\begin{align*}    
   \langle z_1^m, z_1^p z_2^q \rangle_{\mathcal{D}^{(2)}_{\boldsymbol{\mu},1}} 
     &=  \lim _{r \rightarrow 1^{-}} \frac{1}{2 \pi} \iint_{\mathbb{D} \times \mathbb T} m z_1^{m-1}  p \bar{{z}_1}^{p-1} \bar{z_2}^q P_{\mu_1}(z_1) dA(z_1) dt \\
     &= \lim _{r \rightarrow 1^{-}} \frac{1}{2 \pi} \frac{1}{\pi} mp \int_0^1 \int_0^{2\pi} \int_0^{2\pi} r_1^{m+p-1} r_1^{|m-p|}e^{i(m-p)\theta_1} r^{q} e^{-itq}  d\mu_1(\theta_1) dt dr_1.
\end{align*}
    Since $q\geqslant 1,$ it easily follows that $\int_0^{2\pi}e^{-itq} dt=0.$ Therefore, we get that $\langle z_1^m, z_1^p z_2^q \rangle_{\mathcal{D}^{(2)}_{\boldsymbol{\mu},1}}=0.$ On the other hand, using definitions, note that $$\langle z_1^m, z_1^p z_2^q\rangle_{H^2\left(\mathbb{D}^2\right)}=0=\langle z_1^m, z_1^p z_2^q\rangle_{\mathcal{D}^{(2)}_{\boldsymbol{\mu},2}}=\langle z_1^m, z_1^p z_2^q\rangle_{\mathcal{D}^{(2)}_{\boldsymbol{\mu},3}}.$$
    Using \eqref{inner product in 1 variable}, we observe that $\langle 1, z_2^q\rangle_{\mathcal D(\mu_2)}=0.$ This establishes the equality in the lemma in this case.

     \textbf{Case 5:} 
Once again from \eqref{inner product-monomials}, we get that
\begin{align*}
     \langle z_1^m z_2^n, z_1^p z_2^q \rangle_{\mathcal{D}^{(2)}_{\boldsymbol{\mu},1}} 
    &= \lim _{r \rightarrow 1^{-}} \frac{1}{2 \pi} \iint_{\mathbb{D} \times \mathbb T} mz_1^{m-1} z_2^n p \overline{z_1}^{p-1} \overline{z_2}^q P_{\mu_1}(z_1)  dt dA(z_1) \\
    &= \lim _{r \rightarrow 1^{-}} \frac{1}{2 \pi} \frac{1}{\pi} mp \int_0^1 \int_0^{2\pi} \int_0^{2\pi} r_1^{m+p-1} r^{n+q} e^{i(n-q)t}r_1^{|m-p|} e^{i(m-p)\theta_1}  d\mu_1(\theta_1) dt dr_1\\
    &= \delta_{nq} \lim _{r \rightarrow 1^{-}} \frac{1}{2 \pi^2}  mp \int_0^1 \int_0^{2\pi} r_1^{m+p-1} r^{n+q} r_1^{|m-p|} e^{i(m-p)\theta_1}  d\mu_1(\theta_1) dr_1.
\end{align*}
Thus
\begin{eqnarray}\label{first semi-inner product-both variables present}
    \langle z_1^m z_2^n, z_1^p z_2^q \rangle_{\mathcal{D}^{(2)}_{\boldsymbol{\mu},1}}=\begin{cases}
        0 & \mbox{ if }n\neq q\\
        \frac{1}{2\pi} (m \wedge p) \widehat{\mu_1}(p-m) & \mbox{ if } n=q. 
    \end{cases}
\end{eqnarray}
Similarly, we can derive, 
 \begin{eqnarray}\label{second semi-inner product-both variables present}
    \langle z_1^m z_2^n, z_1^p z_2^q \rangle_{\mathcal{D}^{(2)}_{\boldsymbol{\mu},2}}=\begin{cases}
        0 & \mbox{ if }m\neq p\\
        \frac{1}{2\pi} (n \wedge q) \widehat{\mu_2}(q-n) & \mbox{ if } m=p. 
    \end{cases}
\end{eqnarray}
In the following, we use \eqref{inner product-monomials} twice to obtain that 
\begin{eqnarray}\label{third semi-inner product-both variables present}
    \langle z_1^m z_2^n, z_1^p z_2^q \rangle_{\mathcal{D}^{(2)}_{\boldsymbol{\mu},3}}
    &=& \iint_{\mathbb D^2} mn z_1^{m-1} z_2^{n-1}~ pq \overline{z_1}^{p-1} \overline{z_2}^{q-1} P_{\boldsymbol{\mu}}(z_1,z_2) dA(z_1) dA(z_2) \nonumber \\  
    &=&  \frac{1}{\pi^2} mnpq \frac{1}{2( m \vee p)} \frac{1}{2(n \vee q)} \widehat\mu_1(p-m) \widehat \mu_2(q-n)\nonumber \\
    &=& \frac{1}{4\pi^2} (m \wedge p)  (n \wedge q) \widehat\mu_1(p-m) \widehat \mu_2(q-n).
\end{eqnarray}
Using \eqref{inner product in 1 variable}, \eqref{first semi-inner product-both variables present}, \eqref{second semi-inner product-both variables present}, and \eqref{third semi-inner product-both variables present}, we get the result in this case.
\end{proof} 

We note down the case when $m=p$ and $n=q$ of above lemma into the following corollary.
\begin{cor}   
    For any $m,n\in\mathbb Z_{\geqslant 0},$ 
    $\|z_1^{m} z_2^{n}\|_{\boldsymbol{\mu}} = \|z_1^m\|_{\boldsymbol{\mu}} \|z_2^n\|_{\boldsymbol{\mu}}.$
\end{cor}
Lemma \ref{inner product- monomials - in term of single variable} suggests that there should be a relationship between $\mathcal D^{2}(\boldsymbol{\mu})$ and $\mathcal D(\mu_1)\otimes \mathcal D(\mu_2).$ This is indeed the case, as shown in the following theorem.
\begin{theorem}\label{tensor product}
    For any $\mu_1,\mu_2\in\mathcal M_+(\mathbb T),$ the pair of commuting $2$-isometries $(M_{z_1},M_{z_2})$ on $\mathcal D^{2}(\boldsymbol{\mu})$ is unitarily equivalent to the pair $(M_z\otimes I, I\otimes M_z)$ on $(\mathcal D(\mu_1)\otimes \mathcal D(\mu_2)) \oplus (\mathcal D(\mu_1)\otimes \mathcal D(\mu_2)).$ In particular, the pair $(M_{z_1}, M_{z_2})$ is doubly commuting.
\end{theorem}
\begin{proof}
Define $U:\mathcal D^{2}(\boldsymbol{\mu}) \to \mathcal D(\mu_1)\otimes \mathcal D(\mu_2)$ by the rule $U(z_1^mz_2^n):= z^m\otimes z^n,$ for $m,n\in\mathbb Z_{\geqslant 0},$ and extend $U$ linearly. Since the set of polynomials are dense in $\mathcal D^{2}(\boldsymbol{\mu}),$ $U$ extends to whole space $\mathcal D^{2}(\boldsymbol{\mu})$ continuously. 
Using Lemma \ref{inner product- monomials - in term of single variable}, it follows that $U$ is an isometry. 
On the other hand, since $U$ has closed range and $\mathrm{span}\{z^m\otimes z^n:m,n\in\mathbb Z_{\geqslant 0}\}$ is dense in $\mathcal D(\mu_1)\otimes \mathcal D(\mu_2),$ it follows that $U$ is a unitary map. 
Also note that 
$$U M_{z_1}(z_1^mz_2^n)=z^{m+1}\otimes z^n=(M_z\otimes I)U(z_1^m z_2^n)$$ 
and 
$$U M_{z_2}(z_1^mz_2^n)=z^{m}\otimes z^{n+1}=(I\otimes M_z)U(z_1^m z_2^n).$$ 
    This proves that the pair $(M_{z_1},M_{z_2})$ is unitarily equivalent to the pair $(M_z\otimes I, I\otimes M_z).$ This, in particular, implies that $M_{z_1}$ and $M_{z_2}$ on $\mathcal D^{2}(\boldsymbol{\mu})$ is doubly commuting. 
\end{proof}

For the pair of commuting $2$-isometries $(M_{z_1},M_{z_2})$ on $\mathcal{D}^2(\boldsymbol{\mu})$, denote the closed subspace $(\mathcal{D}^2(\boldsymbol{\mu}) \ominus z_1\mathcal{D}^2(\boldsymbol{\mu})) \cap (\mathcal{D}^2(\boldsymbol{\mu}) \ominus z_2 D^2(\boldsymbol{\mu}))$ by $\mathcal{W}$. The subspace $\mathcal W$ plays a major role in this article. With a simple computation, it is also easy to see that $\mathcal{W} = \ker M_{z_1}^* \cap \ker M_{z_2}^*.$ 
Below, we note down some of the important properties of the subspace $\mathcal{W}$.
\begin{rem}
    The proof of Theorem \ref{tensor product} tells us that the kernel function $K$ for $\mathcal D^2(\boldsymbol{\mu})$ is given by 
    \begin{eqnarray*}
        K((z_1,z_2),(w_1,w_2))=K^{\mu_1}(z_1,w_1)K^{\mu_2}(z_2,w_2), \quad (z_1,z_2),(w_1,w_2)\in\mathbb D^2,
    \end{eqnarray*}
    where $K^{\mu_j}$ denotes the kernel function for $\mathcal D(\mu_j),$ $j=1,2.$
\end{rem}
As corollaries of Theorem \ref{tensor product}, we get the following results.
\begin{cor}
    The closed subspace $\mathcal{W}$ is a wandering subspace of $(M_{z_1},M_{z_2})$ on $\mathcal{D}^2(\boldsymbol{\mu})$ and the dimension of $\mathcal{W}$ is $1$.
\end{cor}
\begin{cor}
    The subspace $\ker M_{z_1}^* \subseteq \mathcal{D}^2(\boldsymbol{\mu})$ is a $M_{z_2}$-reducing subspace. Also, the subspace $\ker M_{z_2}^* \subseteq \mathcal{D}^2(\boldsymbol{\mu})$ is a $M_{z_1}$-reducing subspace. 
\end{cor}


Before we proceed further,  for any $f\in \mathcal{D}^2(\boldsymbol{\mu})$ and $z_1, z_2\in\mathbb D,$ we define the operators $L_1$ and $L_2$ by the rules 
\begin{eqnarray}\label{left inverses}
  (L_1f)(z_1,z_2) := \frac{f(z_1,z_2) - f(0,z_2)}{z_1} \quad \text{and} \quad (L_2f)(z_1,z_2) := \frac{f(z_1,z_2)-f(z_1,0)}{z_2}.  
\end{eqnarray}
Using the definition of norm, Lemma \ref{semi-norm of slice functions}, and Theorem \ref{Gleason property at origin}, it turns out that $L_1$ and $L_2$ are indeed in $\mathcal{B}(\mathcal{D}^2(\boldsymbol{\mu}))$. Here, we enlist some of the key properties of these operators. 
\begin{theorem}\label{L_1 commutes with M_{z_2}}
    The operators $L_1$ and $L_2$ are the left inverses of $M_{z_1}$ and $M_{z_2}$ respectively. Moreover, $L_1 M_{z_2} = M_{z_2}L_1$ and $L_2 M_{z_1}= M_{z_1}L_2$.
\end{theorem}
\begin{proof}
    By using the definition of the operators $L_1$, defined above, we have 
\[ 
L_1 M_{z_1}f(z_1,z_2) = \frac{(z_1f)(z_1,z_2)-(z_1f)(0,z_2)}{z_1} = f(z_1,z_2) 
\] 
for all $f\in \mathcal{D}^2(\boldsymbol{\mu})$ and for all $(z_1,z_2) \in \mathbb D^2$. Similarly, we also have $L_2 M_{z_2}= M_{z_2}L_2$. On the other hand, for all $(z_1,z_2) \in \mathbb D^2$, we have
\begin{align*}
    (M_{z_2}L_1f)(z_1,z_2) &= z_2 \bigg \{ \frac{f(z_1,z_2) - f(0,z_2)}{z_1} \bigg\} \\ 
     &= z_2 \Big\{\sum_{m_1=1}^\infty \sum_{m_2=0}^\infty a_{m_1,m_2} z_1^{m_1-1} z_2^{m_2} \Big\} \\
    &= \sum_{m_1=1}^\infty \sum_{m_2=0}^\infty a_{m_1,m_2} z_1^{m_1-1} z_2^{m_2+1} \\
    &= \frac{(z_2f)(z_1,z_2) - (z_2f)(0,z_2)}{z_1} \\
     &=  (L_1 M_{z_2}f)(z_1,z_2)
\end{align*} Analogously, one can also derive $L_2M_{z_1}= M_{z_1}L_2$.  
\end{proof}

\begin{rem}
   Note that, by Theorem \ref{Gleason property at origin}, the operators $M_{z_i} \in \mathcal{B}(\mathcal{D}^2(\boldsymbol{\mu}))$ are bounded below for $i=1,2$, hence $(M_{z_i}^* M_{z_i})^{-1}$ exists as well as $M_{z_i}'$s are left invertible. Moreover, for $i=1,2$, $(M_{z_i}^* M_{z_i})^{-1}M_{z_i}^*$ and the above defined operators $L_i$ are left-inverse of $M_{z_i}$ with $\ker L_i = \ker (M_{z_i}^* M_{z_i})^{-1}M_{z_i}^* = \ker M_{z_i}^*$ and therefore $L_i = (M_{z_i}^* M_{z_i})^{-1}M_{z_i}^*$. 
\end{rem}

\section{Wandering Subspace property for a pair of  analytic 2-isometries}

In this section, we deal with an $n$-tuple of commuting left invertible operators acting on a Hilbert space and prove a multivariate version of S. Shimorin’s result (see \cite[Theorem 3.6]{Shimorin}) for a class of operators acting on a Hilbert space. 

For any left invertible operator $T\in\mathcal B(\mathcal H),$ define $L:=(T^*T)^{-1}T^*.$ Note that $L\in\mathcal B(\mathcal H)$ is the left inverse of $T$ with the property that $\ker L=\ker T^*.$
In the following lemma, we prove that joint kernel of adjoint is non-trivial for any analytic left-inverse commuting tuple.
\begin{lemma}\label{projection and wandering subspace}
     Let $(T_1,\ldots, T_n)$ be a left-inverse commuting $n$-tuple of analytic operators on a Hilbert space $\mathcal{H}$. Then the subspace $\bigcap_{i=1}^n \ker T_i^*$ is a non-trivial closed subspace of $\mathcal{H}$. 
\end{lemma}
   
\begin{proof} 
    For each $1\leqslant i \leqslant n,$ define $L_i:=(T_i^*T_i)^{-1}T_i^*.$ Since $T_1$ is analytic and left invertible, we can conclude that $\ker T_1^*$ is non-trivial.
    Let $x_0 \in \ker T_1^*$.  Since $L_1x_0 =0$, it follows that $T_2L_1x_0 = 0$. By using the hypothesis, we have $L_1T_2 x_0 = 0$, which implies $T_2x_0 \in \ker T_1^*$. 
    Hence, $\ker T_1^*$ is invariant under $T_2$. On the other hand, using the commutativity of $T_1$ and $T_2$, it is straightforward to see that $\ker T_1^*$ is invariant under $T_2^*$. Hence, $\ker T_1^*$ is a $T_2$-reducing subspace. Since $T_2 \big \vert_{\ker T_1^*}: \ker T_1^* \to \ker T_1^*$ is an analytic operator, that is, 
    \[ 
    \bigcap_{n \geqslant 0} \big(T_2 \big \vert_{\ker T_1^*}\big)^n(\ker T_1^*) = \{0\},
    \] 
    therefore, $T_2 \big \vert_{\ker T_1^*}$ cannot be onto. Moreover, since $T_2 \big \vert_{\ker T_1^*}$ is also left-invertible, it's range is closed. Therefore, the subspace $\ker (T_2^* \big \vert_{\ker T_1^*}) \subseteq  \ker T_1^* $ is non-trivial. Hence, $\ker T_1^* \bigcap \ker T_2^* \neq \{0\}$. Repeating the analogous argument a finite number of times we conclude that $\bigcap_{i=1}^n \ker T_i^* $ is non-trivial.
    \end{proof}


From the proof of the above lemma, the following result is evident. This generalizes \cite[Proposition 2.2]{ABJS}.
\begin{cor}\label{reducing condition}
Let $\boldsymbol T=(T_1,\ldots,T_n) \in \mathcal{B}(\mathcal{H})^n$ be any commuting $n$-tuple of bounded linear operators. For any $i,j \in \{1,\ldots,n\}$ with $i\neq j$, the following are equivalent:
\begin{enumerate}
    \item[(a)] $L_iT_j=T_jL_i$ 
    \item[(b)] $\ker T_i^*$ is a $T_j$-reducing subspace of $\mathcal{H}$.  
\end{enumerate}
\end{cor}

Now, we proceed to prove the existence of the wandering subspace property for a class of left-inverse commuting $n$-tuple of bounded linear operators. This generalizes many existing results e.g. see \cite{Ma, RT, SSW, ABJS}. 

{Here, we would also like to take this opportunity to point out that conditions (analytic or algebraic) on $\ker T^*$ have played crucial role in modelling a class of bounded linear operators on a Hilbert space, for instance see \cite[Theorem 2.5]{ACJS}.}
\begin{theorem}\label{wandering subspace property-abstract pair}
Let $\boldsymbol T=(T_1,\ldots,T_n)$ be a left-inverse commuting tuple on a Hilbert space $\mathcal{H}$. Then the tuple $\boldsymbol T$ has individual wandering subspace property if and only if it has the joint wandering subspace property. Moreover, $\mathcal{W} = \bigcap_{i=1}^n \ker T_i^*$ is the corresponding generating wandering subspace for the $n$-tuple $\boldsymbol T$ on $\mathcal{H}$. 
\end{theorem}
\begin{proof}
We first assume that the tuple $\boldsymbol T$ has individual wandering subspace property. 
For the subspace $\mathcal{W} = \bigcap_{i=1}^n \ker T_i^* \subseteq \mathcal{H}$, one sided inclusion is trivial, that is, 
    \[
    \bigvee_{k_1,\ldots,k_n \geqslant 0} T_1^{k_1} \cdots T_n^{k_n} \mathcal W \subseteq \mathcal{H}. 
    \]  
 To prove the other inclusion, note that, by hypothesis, we know that the operator $T_1$ on $\mathcal{H}$ has the wandering subspace property. Therefore, we have $\mathcal{H} = \bigvee_{k_1 \geqslant 0} T_1^{k_1} (\ker{T_1^*})$. 
 On the other hand, since the $n$-tuple $\boldsymbol T$ is a left-inverse commuting tuple, by Corollary \ref{reducing condition} the subspace $\ker{T_1^*}$ is a $T_2$-reducing subspace. Therefore, we have 
 \[ 
 \ker{T_1^*} = P_{\ker{T_1^*}} ( \bigvee_{k_2 \geqslant 0} T_2^{k_2} \ker{T_2^*} ) 
             = \bigvee_{k_2 \geqslant 0} T_2^{k_2} (\ker{T_1^*} \bigcap \ker{T_2^*}).
                 \]
 Combining all these we have 

 \[
 \mathcal{H}  = \bigvee_{k_1,k_2 \geqslant 0} T_1^{k_1} T_2^{k_2} (\ker{T_1^*} \bigcap \ker{T_2^*}).
 \]
 For some $l \in \{1,\ldots,n\}$, let the corresponding tuple $(T_1,\ldots,T_l)$ has the wandering subspace property, that is, 
 \begin{eqnarray}\label{for converse}
      \mathcal{H} = \bigvee_{k_1,\ldots,k_l \geqslant 0} T_1^{k_1} \cdots T_l^{k_l} \Big( \bigcap_{i=1}^l \ker{T_i^*}\Big).
 \end{eqnarray}
 Again, by applying Corollary \ref{reducing condition}, the subspace $\bigcap_{i=1}^l \ker{T_i^*}$ is $T_{l+1}$-reducing. As a result, we have
 \[
 \bigcap_{i=1}^l \ker{T_i^*} = P_{\bigcap_{i=1}^l \ker{T_i^*}} \Big( \bigvee_{k_{l+1} \geqslant 0} T_{l+1}^{k_{l+1}} \ker{T_{l+1}^*} \Big) 
             = \bigvee_{k_{l+1} \geqslant 0} T_{l+1}^{k_{l+1}} \Big(\bigcap_{i=1}^{l+1} \ker{T_i^*} \Big).
 \]
 Hence, the tuple $(T_1,\ldots,T_{l+1})$ also has the wandering subspace property with the wandering subspace $\bigcap_{i=1}^{l+1} \ker{T_i^*}$. Therefore, by induction, we conclude that the given tuple $T=(T_1,\ldots,T_n)$ has the wandering subspace property with the generating wandering subspace $\bigcap_{i=1}^n \ker{T_i^*}$.

 To see the converse part, we assume $\boldsymbol T$ has the joint wandering subspace property. We claim that for each $1\leqslant j\leqslant n,$ $T_j$ has the wandering subspace property. This can easily be seen by applying $P_{\ker T_j^*}$ in \eqref{for converse} for $l=n$ followed by an application of Corollary \ref{reducing condition}. 
\end{proof}

\begin{rem}
The proof above indicates that Theorem \ref{wandering subspace property-abstract pair} is also true with much lesser hypothesis. Specifically, statement of Theorem \ref{wandering subspace property-abstract pair} is valid for any $n$-tuple of commuting left-invertible operators $(T_1,\ldots,T_n)$ acting on a Hilbert space $\mathcal{H},$ provided each $T_j$ has the wandering subspace property, and for each fixed $j \in \{2,\ldots,n\},$ 
\[ 
L_i T_j = T_j L_i \quad \quad \text{for all } i = 1,\ldots,j-1.
\]
On the other hand, once we have the conclusion of Theorem \ref{wandering subspace property-abstract pair}, we can easily derive the remaining operator identities namely  for each fixed $j \in \{2,\ldots,n\},$
\[ 
L_i T_j = T_j L_i \quad \quad \text{for all } i = j+1,\ldots,n.
\]
In other words, having lesser number of operator identities in the hypothesis of Theorem \ref{wandering subspace property-abstract pair} does not expand the class of commuting tuple of the operators having wandering subspace property within this set up. This is why we  choose to write the statement of Theorem \ref{wandering subspace property-abstract pair} with full set of operator identities $L_i T_j = T_j L_i, \ 1\leqslant i \neq j \leqslant n.
$ 
\end{rem}

The following example presents a class of commuting \(n\)-tuples of bounded linear operators that are not doubly commuting but are left-inverse commuting tuples.
\begin{example}\label{counter example - weighted shift}
Suppose $\boldsymbol T = (T_1, \ldots, T_n)$ is a commuting operator-valued multishift on $\mathcal H = \oplus_{\alpha \in \mathbb Z_{\geqslant 0}^n} H_\alpha$ with invertible operator weights $\big\{A^{(j)}_{\alpha}: \alpha \in \mathbb Z_{\geqslant 0}^n,\  j=1, \ldots, n\big\}$. For the definition and other properties, we refer the reader to \cite{L, LT, JL, JJS, H, GKT-1, GKT-2}. Using the action of $T_j^*$ and $T_i,$ one can easily verify that $T_j^*T_i = T_iT_j^*$ if and only if     
\begin{equation}\label{doubly_commuting_operator_valued_weighted_shift}
A^{(i)}_{\alpha + \varepsilon_j}(A^{(j)}_{\alpha} A^{(j)*}_{\alpha})A^{(i)^{-1}}_{\alpha + \varepsilon_j} = A^{(j)}_{\alpha + \varepsilon_i}A^{(j)*}_{\alpha + \varepsilon_i}, \quad 1 \leqslant i,j \leqslant n, \alpha \in \mathbb Z_{\geqslant 0}^n,
\end{equation}
where for each $1 \leqslant i \leqslant n,$ $\varepsilon_i:=(0,\ldots,0,1,0,\ldots,0)\in\mathbb Z_{\geqslant 0}^n$ with $1$ at the $i$-th place and $0$ elsewhere.
For each $1 \leqslant i \leqslant n,$ one can easily compute $L_i,$ the left inverse of $T_i$ with $\ker L_i = \ker T_i^*,$ to be given by 
\begin{equation}\label{left_inverse_for_operator_valued_weighted_shift}
L_i\left(\oplus_{\alpha \in \mathbb N^n}x_{\alpha}\right)= \oplus_{\alpha \in \mathbb N^n} A^{(j)^{-1}}_{\alpha} x_{\alpha + \varepsilon_i}, \quad \oplus_{\alpha \in \mathbb Z_{\geqslant 0}^n} x_{\alpha} \in \mathcal H,\ i = 1, \ldots, n.
\end{equation}
From \eqref{left_inverse_for_operator_valued_weighted_shift} and the action of $T_j,$ it easily follows that for any $1 \leqslant i,j \leqslant n,$ the equality $L_iT_j = T_jL_i$ holds if and only if 
\begin{equation}\label{left_inverse_commuting_operator_valued_weighted_shift}
A^{(j)}_{\alpha + \varepsilon_i}A^{(i)}_{\alpha} = A^{(i)}_{\alpha + \varepsilon_j}A^{(j)}_{\alpha}, \quad \alpha \in \mathbb Z_{\geqslant 0}^n.
\end{equation}
Equality in \eqref{left_inverse_commuting_operator_valued_weighted_shift} is precisely the condition for commutativity of the tuple $(T_1, \ldots, T_n).$ This discussion proves that the class of operator-valued multishifts with invertible operator weights is always a class of left-inverse commuting tuples, whereas it is rarely doubly commuting. For instance, if we take $A^{(j)}_{\alpha} = w^{(j)}_{\alpha} I$ where $w^{(j)}_{\alpha} \neq 0,$ then \eqref{doubly_commuting_operator_valued_weighted_shift} simply becomes 
\begin{equation}\label{doubly commuting for scalar valued multi-shift}
|w^{(j)}_{\alpha}| = |w^{(j)}_{\alpha + \varepsilon_i}|, \quad 1 \leqslant i,j \leqslant n, \alpha \in \mathbb Z_{\geqslant 0}^n.
\end{equation}
 Equation \eqref{doubly commuting for scalar valued multi-shift} indicates that the class of weighted multi-shift which are doubly commuting is extremely small whereas each of them is left-inverse commuting tuple.
\end{example}

\section{Model for a class of pair of cyclic analytic 2-isometries}

In this section, we prove the following model theorem which shows that the class of multiplication operators $(M_{z_1},M_{z_2})$ on $\mathcal D^2(\boldsymbol{\mu})$ completely characterizes that of left-inverse commuting pair of analytic $2$-isometries which satisfy certain splitting properties. With the help of this theorem, we give an example of a left-inverse commuting pair of analytic $2$-isometries which is not a doubly commuting pair.
\begin{theorem}\label{Model theorem}
   Let $(T_1,T_2)$ be a pair of cyclic, analytic, 2-isometries on a Hilbert space $\mathcal{H}$ with unit cyclic vector $x_0 \in \ker T_1^*\cap \ker T_2^*$. Then $(T_1,T_2)$ on $\mathcal{H}$ is unitarily equivalent with $(M_{z_1}, M_{z_2})$ on $\mathcal{D}^{(2)}(\boldsymbol{\mu})$ if and only if $L_1 T_2 = T_2 L_1$, $L_2 T_1 = T_1 L_2$ and for all $m,n \in \mathbb N \setminus \{0\}$ 
   \begin{align*}\label{splitting-T-2}
   &\langle T_1^m T_2^n x_0, T_1 T_2 x_0 \rangle = \langle T_1^mx_0, T_1x_0 \rangle \langle T_2^nx_0, T_2x_0 \rangle,\\ 
   & \langle T_1^mT_2x_0, T_1 T_2^nx_0 \rangle = \langle T_1^mx_0, T_1x_0 \rangle \langle T_2x_0, T_2^nx_0 \rangle.
   \end{align*}
\end{theorem}
\begin{proof}
    Since $(T_1,T_2)$  is an analytic left-inverse commuting pair, it follows from Theorem \ref{wandering subspace property-abstract pair} that $\ker T_1^*\cap \ker T_2^*$ is wandering subspace for $(T_1,T_2).$   
    Using the hypothesis that $(T_1,T_2)$ is a cyclic pair, it follows that $ \ker T_1^*\cap \ker T_2^*=\mathrm{span}\{x_0\},$ for some unit vector $x_0.$ 
    Consider the operator $S_1:=T_1|_{\ker T_2^*}$ and $S_2:=T_2|_{\ker T_1^*}.$ From Lemma \ref{projection and wandering subspace}, it follows that $S_1$ and $S_2$ are bounded linear operators on $\ker T_2^*$ and $\ker T_1^*$ respectively. Now it follows that $S_1$ and $S_2$ are cyclic analytic $2$-isometries with cyclic subspace $\ker T_1^*\cap \ker T_2^*.$  
    Using \cite[Theorem 5.1]{Richter1}, we have $\mu_1, \mu_2\in\mathcal M_+(\mathbb T)$ and unitary maps $U_1: \ker T_2^* \rightarrow \mathcal{D}(\mu_1)$ and $U_2: \ker T_1^* \rightarrow \mathcal{D}(\mu_2)$ such that
    \[ 
    U_i S_i = M_z U_i ~~\quad ~~\text{and}~~ \quad U_i x_0=1.
    \]
     Let $\boldsymbol{\mu}:= \mu_1 \times \mu_2  $ be the product measure and consider the corresponding Dirichlet-type spaces 
    $ \mathcal{D}^2(\boldsymbol{\mu})$. We define a linear map $U: \mathcal{H} \rightarrow \mathcal{D}^2(\boldsymbol{\mu})$, which is defined by 
    \[ 
    U(p(T_1, T_2)x_0) = p(z_1,z_2),
    \]
    where $p(T_1, T_2)x_0 = \sum_{k_1,k_2=0}^{n_1,n_2} a_{k_1 k_2} T_1^{k_1} T_2^{k_2} x_0$ for any complex valued polynomial $p(z_1,z_2)= \sum_{k_1,k_2=0}^{n_1,n_2}  a_{k_1 k_2} z_1^{k_1} z_2^{k_2} $. We claim that the map defines a unitary map from $\mathcal{H}$ onto $\mathcal{D}^2(\boldsymbol{\mu})$. So we need to show that $\| P(T_1, T_2)x_0 \|^2 = \|P(z_1,z_2)\|^2_{\mathcal{D}^2(\mu)}$. Therefore, it is enough to verify that 
    \[
    \langle T_1^m T_2^n x_0, T_1^p T_2^q x_0 \rangle = \langle z_1^mz_2^n, z_1^p z_2^q \rangle_{\mathcal{D}^2(\mu)} \quad~~\text{for all $m,n,p,q \in \mathbb N.$}
    \]
    First note that, in above, if $n,q=0$ then for any $m,p \in \mathbb N$, considering the above unitary map $U_1$, we have
\begin{align*}
    \langle T_1^m x_0, T_1^p x_0 \rangle &= \langle U_1 T_1^m x_0, U_1 T_1^p x_0 \rangle_{\mathcal{D}(\mu_1)} \nonumber \\
    &= \langle M_z^m U_1 x_0, M_z^p U_1 x_0 \rangle_{\mathcal{D}(\mu_1)} \nonumber \\
    &= \langle z^m, z^p \rangle_{\mathcal{D}(\mu_1)} \nonumber \\
    &= \langle z_1^m, z_1^p \rangle_{\mathcal{D}^2(\boldsymbol{\mu})}. 
\end{align*}
    Similarly, we also have $\langle T_2^n x_0, T_2^q x_0 \rangle = \langle z_2^n, z_2^q \rangle_{\mathcal{D}^2(\boldsymbol{\mu})}$ for any $n,q \in \mathbb N$. On the other hand, by Lemma \ref{projection and wandering subspace}, we know that for any $i,j =1,2$ with $i \neq j$ the subspace $KerT_j^*$ is $T_i$ reducing. Without loss of generality, it is enough to prove for $i=1$ and $j=2$. By using this fact it is obvious that 
    \[ 
    \langle T_1^m x_0, T_1^p T_2^q x_0 \rangle = 0 \quad \quad~~\text{for any $q \in \mathbb N \setminus \{0\}$ and $m,p \in \mathbb N.$}
    \] 
    Now by combining the above equation and the Lemma \ref{inner product- monomials - in term of single variable} we conclude 
    \[ 
    \langle T_1^m x_0, T_1^p T_2^q x_0 \rangle = \langle z_1^m, z_1^p z_2^q \rangle_{\mathcal{D}^2(\boldsymbol{\mu})} = 0 \quad \quad~~\text{for any $q \in \mathbb N \setminus \{0\}$ and $m,p \in \mathbb N.$}
    \]
    It remains to deal with the case where $m,n,p,q \in \mathbb N \setminus \{0\}$. Again, without loss of generality, it is enough to establish the equation for $m \geqslant p, n \geqslant q$ and $m \geqslant p, n \leqslant q$.
    First we consider the case $m \geqslant p, n \geqslant q$. Since $T_1$ is a $2$-isometry therefore we deduce 
\begin{align*}
    \langle T_1^m T_2^n x_0, T_1^p T_2^q x_0 \rangle = (m \wedge p)  \langle T_1^{m-p+1} T_2^nx_0, T_1T_2^q x_0 \rangle.
\end{align*}
    Analogously, using the fact that $T_2$ is a $2$-isometry, we reduce the above term further and it becomes 
\begin{align*}
    (m \wedge p)  \langle T_1^{m-p+1} T_2^nx_0, T_1T_2^q x_0 \rangle 
    &= (m \wedge p) (n \wedge q) \langle T_1^{m-p+1} T_2^{n-q+1}x_0, T_1T_2x_0 \rangle.  
\end{align*}
    Combining all these we have 
\begin{align*}
     \langle T_1^m T_2^n x_0, T_1^p T_2^q x_0 \rangle = (m \wedge p) (n \wedge q) \langle T_1^{m-p+1} T_2^{n-q+1}x_0, T_1T_2x_0 \rangle.
\end{align*}
    By the given hypothesis, the above identities on $T_1$ and $T_2$ individually and the Lemma \ref{inner product- monomials - in term of single variable} we conclude 
\begin{align*}
    \langle T_1^m T_2^n x_0, T_1^p T_2^q x_0 \rangle &= (m \wedge p) (n \wedge q) \langle T_1^{m-p+1} T_2^{n-q+1}x_0, T_1T_2x_0 \rangle \\
    &= (m \wedge p) (n \wedge q) \langle T_1^{m-p+1} x_0, T_1 x_0 \rangle \langle T_2^{n-q+1} x_0, T_2x_0 \rangle \\
    &= (m \wedge p) (n \wedge q) \langle z_1^{m-p+1}, z_1 \rangle_{\mathcal{D}^2(\boldsymbol{\mu})} \langle z_2^{n-q+1}, z_2 \rangle_{\mathcal{D}^2(\boldsymbol{\mu})} \\
    &= \langle z_1^m z_2^n, z_1^p z_2^q \rangle_{\mathcal{D}^2(\boldsymbol{\mu})}. 
\end{align*}   
    Similarly, for the second case with $m \geqslant p, n \leqslant q$ by using the given hypothesis and the Lemma \ref{inner product- monomials - in term of single variable} we deduce 
    \begin{align*}
    \langle T_1^m T_2^n x_0, T_1^p T_2^q x_0 \rangle &=(m \wedge p) (n \wedge q) \langle T_1^{m-p+1} T_2x_0, T_1T_2^{q-n+1}x_0\rangle \\
    &= (m \wedge p) (n \wedge q) \langle T_1^{m-p+1} x_0, T_1 x_0 \rangle \langle T_2 x_0, T_2^{n-q+1} x_0 \rangle \\
    &= (m \wedge p) (n \wedge q) \langle z_1^{m-p+1}, z_1 \rangle_{\mathcal{D}^2(\boldsymbol{\mu})} \langle z_2, z_2^{n-q+1} \rangle_{\mathcal{D}^2(\boldsymbol{\mu})} \\
    &= \langle z_1^m z_2^n, z_1^p z_2^q \rangle_{\mathcal{D}^2(\boldsymbol{\mu})}. 
    \end{align*}
The converse follows as an application of Lemma \ref{inner product- monomials - in term of single variable} and Proposition \ref{L_1 commutes with M_{z_2}}.
\end{proof}

\begin{rem}
    For any $\boldsymbol{\mu},\boldsymbol{\nu}\in\mathcal P \mathcal M_+(\mathbb T^2),$ we content that $(M_{z_1},M_{z_2})$ on $\mathcal D^2(\boldsymbol{\mu})$ is unitarily equivalent to $(M_{z_1},M_{z_2})$ on $\mathcal D^2(\boldsymbol{\nu})$ if and only if $\boldsymbol{\mu}=\boldsymbol{\nu}.$ To prove this, note that $M_{z_1}$ and $M_{z_2}$ are analytic $2$-isometries on $\mathcal D(\mu_j)$ and $\mathcal D(\nu_j)$ for each $j=1,2.$ Therefore if $U$ is the unitary (as in the proof of Theorem \ref{Model theorem}) intertwining pair $(M_{z_1},M_{z_2})$ on  $\mathcal D^2(\boldsymbol{\mu})$ and the pair $(M_{z_1},M_{z_2})$ on  $\mathcal D^2(\boldsymbol{\nu})$ then $M_{z_1}|_{\mathcal D(\mu_1)}$ and $M_{z_1}|_{\mathcal D(\nu_1)}$ is intertwined by the unitary map $U|_{\mathcal D(\mu_1)}.$ Therefore, from \cite[Theorem 5.2]{Richter1}, it follows that $\mu_1=\nu_1.$ Similarly, one can prove $\mu_2=\nu_2.$
\end{rem}
    
   We conclude this section with the following example which also exhibits the importance of Theorem \ref{wandering subspace property-abstract pair}.  
\begin{example} \label{Example-not doubly commuting}
Let $\mu_1,\mu_2\in\mathcal M_+(\mathbb T).$ Consider the space $\mathcal D(\mu_1,\mu_2)$ as introduced in \cite{Sameer et. al.}. We content that the pair $(M_{z_1}, M_{z_2})$ on $\mathcal D(\mu_1,\mu_2)$ is a left-inverse commuting pair of analytic $2$-isometries. 

To see this, we first note that from \cite[Corollary 3.8]{Sameer et. al.} the pair $(M_{z_1}, M_{z_2})$ is a toral $2$-isometry and hence $2$-isometric pair. 
Note that for any $f\in \mathcal D(\mu_1,\mu_2)$, the function $f(0,z_2)$ is in $\mathcal D(\mu_1,\mu_2),$ see Lemma \ref{semi-norm of slice functions}. 

Define $g(z_1,z_2):=f(z_1,z_2)-f(0,z_2).$ 
Note that $g\in \mathcal D(\mu_1,\mu_2)$ and $\{(z_1,z_2)\in\mathbb D^2: z_1=0\}$ is in the zero set of $g.$ 
Now using \cite[Theorem 2.2]{Sameer et. al.}, we observe that $g/z_1$ (this is same as $L_1(f)$, see \eqref{left inverses}) is in $\mathcal D(\mu_1,\mu_2).$ The operator $L_1$ serves as a left inverse of $M_{z_1}$ on $\mathcal D(\mu_1,\mu_2).$ Rerunning the proof of Theorem \ref{L_1 commutes with M_{z_2}} proves that $(M_{z_1},M_{z_2})$ is a left-inverse commuting pair.
By Theorem \ref{wandering subspace property-abstract pair}, the pair $(M_{z_1}, M_{z_2})$ on $\mathcal D(\mu_1,\mu_2)$ possess wandering subspace property. 

Now we assume that each of  $\mu_1$ and $\mu_2$ is normalized Lebesgue measure $\sigma$ on the circle $\mathbb T$. 
With the help of the fact that the set of monomials forms an orthogonal basis for $\mathcal D(\sigma,\sigma)$, it follows from a simple computation that for any $p \geqslant 1$ and $q \geqslant 0$
\[ 
M_{z_1}^*(z_1^p z_2^q) = \frac{p+q+1}{p+q} z_1^{p-1} z_2^q.
\]
It helps us to obtain that
\[
M_{z_2}M_{z_1}^*(z_1^p z_2^q) = \frac{p+q+1}{p+q} z_1^{p-1} z_2^{q+1} \neq \frac{p+q+2}{p+q+1} z_1^{p-1} z_2^{q+1} = M_{z_1}^*M_{z_2}(z_1^p z_2^q).
\]
Therefore, we can conclude that the pair $(M_{z_1},M_{z_2})$ on $\mathcal D(\mu_1,\mu_2)$, where $\mu_1$ and $\mu_2$ are normalized Lebesgue measures on the circle $\mathbb T$, is not a doubly commuting pair. 
\end{example}

 \textbf{Acknowledgements:}  
 The first and second named author's research are supported by the DST-INSPIRE Faculty Grant with Fellowship No. DST/INSPIRE/04/2020/001250
 and DST/INSPIRE/04/2017/002367 respectively. The third named author's research work is supported by the BITS Pilani institute fellowship. 
 The authors are deeply grateful to Prof. Sameer Chavan for his numerous fruitful discussions and insightful comments, which significantly improved the quality and presentation of this article. They would also like to express their appreciation to Dr. Soumitra Ghara and Dr. Ramiz Reza for their valuable observations and feedback on the content.
 The authors also had the opportunity to present this work at the conference ``Operator Analysis: A renaissance (2024)", in Gandhinagar in July 2024, which was funded by Prof. Charian Varughese. Authors are grateful for the support and the valuable feedback received from the mathematicians in attendance.

\end{document}